\documentclass{article}

\usepackage[margin=1in]{geometry}

\usepackage{amsfonts}
\usepackage{amsmath}
\usepackage{amsthm}
\usepackage{graphicx}
\usepackage{epstopdf}
\usepackage{algorithmic}
\usepackage{enumitem}
\usepackage{comment}
\usepackage{url}
\usepackage{color}
\usepackage{todonotes}
\usepackage[normalem]{ulem}

\newtheorem{remark}{Remark}
\newtheorem{definition}{Definition}
\newtheorem{example}{Example}
\newtheorem{lemma}{Lemma}
\newtheorem{proposition}{Proposition}
\newtheorem{theorem}{Theorem}

\newtheorem{corollary}{Corollary}

\newtheorem{assumption}{Assumption}

\newcommand{\rv}[1]{\textcolor{black}{#1}}
\newenvironment{rvn}{\color{black}}{}

\newcommand{\spsd}{\succeq 0}

\newcommand{\spd}{\succ 0}
\newcommand{\range}{\mathrm{range}}

\newcommand{\T}{\mathrm{T}}
\newcommand{\st}{\mathcal{S}^3}

\title{Classification of Double Saddle-Point Systems\thanks{The University of British Columbia, 
  \texttt{\{smbrad,greif\}@cs.ubc.ca}}}

\author{Susanne Bradley \and Chen Greif}
\date{}

\newcommand{\sA}{\mathcal{A}}
\newcommand{\sD}{\mathcal{D}}
\newcommand{\sK}{\mathcal{K}}
\newcommand{\sL}{\mathcal{L}}

\newcommand{\sT}{\mathcal{T}}
\newcommand{\sP}{\mathcal{P}}

\begin{document}

\maketitle

\begin{abstract}
We offer a classification of a broad and practically relevant class of symmetric double saddle-point system. At the core of the paper is the division  of the associated matrices into ``block-arrow'' and ``block-tridiagonal'' forms. We describe relevant applications, invertibility conditions, spectral properties, and  block preconditioners. Our discussion is kept within a general framework rather than tailored to specific applications.
\end{abstract}



\section{Introduction}
\label{sec:intro}
Double (or twofold) saddle-point systems are linear systems whose associated matrices can be written in a $3 \times 3$-block form, with specific structural properties. These systems have attracted growing attention since the early 2000s, with a noticeable proliferation of work over the past decade, driven by their appearance in diverse applications and the challenges they pose for numerical solution. These systems are difficult to solve efficiently because of their indefiniteness, strong block coupling, conditioning issues, and the need to preserve properties of the underlying operators. And yet, despite substantial research on the topic, their formal definition and systematic characterization remain incomplete in the literature. This gap motivates the present work, whose primary goal is to provide a clear and coherent classification of these linear systems.

The term ``twofold saddle-point problem'' first appeared in the early 2000s in the work of Gatica and collaborators, who introduced and analyzed such formulations, and provided solvability and error estimates \cite{BarrientosGaticaStephan2001,GaticaHeuer2002,GaticaHeuerMeddahi2003,GaticaMeddahi2000}. The alternative term  ``double saddle-point'' (in the specific context of block linear systems) first appears in the 2011 paper of Howell \& Walkington \cite{hw11}. Other early uses of the term were by Ramage \& Gartland to describe liquid-crystal director models \cite{GartlandRamage2015,rg13}, by Planas Badenas to describe incompressible magnetohydrodynamics \cite{p13}, and by Kolmbauer and Langer to describe time-harmonic eddy current problems \cite{kl13}. Beik \& Benzi used the same term and investigated iterative methods for these systems  \cite{BeikBenzi2018}. As is common for terminology that develops across multiple application communities, these early uses refer to systems with differing algebraic structures and levels of symmetry. Over time, however, the term has come to be used more narrowly to describe a specific class of structured saddle-point systems. Significant work has been done on numerical solution methods for these systems, including block preconditioning; see, e.g., \cite{BalaniHajarianBergamaschi2024,BergamaschiTriangular2025,bmpp24,sz18} and the references therein. 
 
We define a broad class of double saddle-point matrices as those that have one of the following canonical forms (or can be symmetrically permuted into one of these forms):
\begin{equation}
\label{eq:2sp_arrow_ns}
   \sA_{\text{ns}} = \left[ \begin{array}{ccc}
    A_1 & B_3^T & B_4^T \\
    B_1 & -A_2 & 0 \\
    B_2 & 0 & -A_3
    \end{array} \right] 
\quad {\rm or} \quad
    \sT_{\text{ns}} = \left[ \begin{array}{ccc}
    A_1 & B_3^T & 0 \\
    B_1 & -A_2 & B_4^T \\
    0 & B_2 & A_3
    \end{array}\right],
\end{equation}
where the notation $\sA$ denotes ``arrow'' and $\sT$ denotes ``tridiagonal,'' referring to the block structure of these matrices.
The subscript ``ns'' denotes ``nonsymmetric.''  We assume that the (1,1) block $A_1$ is at least as large in size as the (2,2) and (3,3) blocks, and that the symmetric parts of the diagonal blocks $A_1, A_2,$ and $A_3$ are positive semidefinite.
These conditions alone are not restrictive enough; for example, they do not exclude trivial cases such as the zero matrix. To better reflect structures encountered in practice, we will narrow in on a particular family of double saddle-point matrices, which satisfy additional requirements.

One can easily observe differences between the two forms in \eqref{eq:2sp_arrow_ns}—for example, the differing placements of the $B$ blocks and the sign of $A_3$. These differences  reflect variations in the literature, where matrices with different layouts have been referred to as ``double saddle-point systems.” The common thread and a reason for this terminology is that both forms can be expressed in the structure of a {\em classical} saddle-point system. 
To make this connection concrete, we recall the following definition:

\begin{definition}\cite[Eq. (1.1)--(1.2)]{bgl05}
\label{defn:sp}
A saddle-point system is a $2 \times 2$-block linear system of the form
$$ \begin{bmatrix} A & B_1^T \\ B_2 & -C \end{bmatrix} \begin{bmatrix} x \\ y \end{bmatrix} = \begin{bmatrix} f \\ g \end{bmatrix},$$ where $A \in {\mathbb R}^{n \times n}$, $B_1, \ B_2 \in {\mathbb R}^{m \times n}$, $C \in {\mathbb R}^{m \times m}$, and $n \ge m.$
\end{definition}

The following partitioning reveals the classical saddle-point structure of $\sA_{\text{ns}}$:
\[
\sA_{\text{ns}} = \left[\begin{array}{c|cc}
A_1 & B_3^T & B_4^T \\
\hline
B_1 & -A_2 & 0 \\
B_2 & 0 & -A_3
\end{array}\right].
\]
The matrix $\sT_{\text{ns}}$ admits a classical saddle-point formulation after symmetric permutation of the second and third block-rows/columns. Specifically, if $A_1 \in {\mathbb R}^{n_1 \times n_1}, A_2 \in {\mathbb R}^{n_2 \times n_2}, A_3 \in {\mathbb R}^{n_3 \times n_3}$, we define the block permutation
\[
\Pi_{(23)} :=
\begin{bmatrix}
I_{n_1} & 0        & 0 \\
0       & 0        & I_{n_3} \\
0       & I_{n_2}  & 0
\end{bmatrix}.
\]
Then, the symmetrically block-permuted matrix can be written as a saddle-point matrix whose leading block is itself block diagonal:
\[
\sK_{\mathrm{ns}} := \Pi_{(23)}^{T}\, \sT_{\mathrm{ns}}\, \Pi_{(23)} = \left[\begin{array}{cc|c}
A_1 & 0 & B_3^T \\
0 & A_3 & B_2 \\
\hline
B_1 & B_4^T & -A_2
\end{array}\right].
\] 

Our characterization of double saddle-point systems thus extends Definition \ref{defn:sp}. 
Because double saddle-point matrices have nine blocks, covering all possible block configurations and attributes would yield an unworkably long list. For this reason, we focus on a problem class that, while narrower, is of central importance, and we provide a  unified description of its mathematical and numerical properties.

 \begin{definition}
We define the family $\st$  as the set of symmetric nonsingular $3 \times 3$-block matrices 
\begin{equation}
\label{eq:2sp}
   \sA = \left[ \begin{array}{ccc}
    A_1 & B_1^T & B_2^T \\
    B_1 & -A_2 & 0 \\
    B_2 & 0 & -A_3
    \end{array} \right] 
\quad {\rm or} \quad
    \sT = \left[ \begin{array}{ccc}
    A_1 & B_1^T & 0 \\
    B_1 & -A_2 & B_2^T \\
    0 & B_2 & A_3
    \end{array}\right],
\end{equation}
where $A_1 \in {\mathbb R}^{n_1 \times n_1}$, $A_2 \in {\mathbb R}^{n_2 \times n_2}$, and $A_3 \in {\mathbb R}^{n_3 \times n_3}$ for either of the two forms, and  the dimensions of $B_1$ and $B_2$ follow accordingly (and are different for each of those forms). The matrices $\sA$  and $\sT$ satisfy the following conditions:
\begin{enumerate}[label={[C\arabic*]}]
    \item \label{[C1]} \rv{$A_1$ is at least as large in size as $A_2$ and $A_3$}, and is symmetric positive definite.
    \item \label{[C2]} \rv{$A_2$ and $A_3$} are symmetric positive semidefinite.
    \item \label{[C3]} $B_1$ and $B_2$ either have full rank or have a small nullity (e.g., 1).
\end{enumerate}
\label{defn:dsp}
\end{definition}

The size and positive definiteness restriction \ref{[C1]} on the leading block in general  prevent us from converting $\sA$ to $\sT$ and vice versa using simple symmetric permutations and block negations; it is therefore necessary to refer to them as two distinct forms of block matrices. We also note that it is common to assume for $\sA$ that $n_1 \ge n_2+n_3$ and for $\sT$ that $n_1 \ge n_2 \ge n_3$; see Lemma \ref{lem:dimr} for a justification.\textbf{}

\begin{remark}
One could argue that matrices of the form
\begin{equation}
    \label{eq:K}
\sK  = \left[\begin{array}{ccc}
A_1 & 0 & B_1^T \\
0 & A_2 & B_2^T \\
B_1 & B_2 & -A_3
\end{array}\right],
\end{equation}
i.e., the symmetric counterpart of $\sK_{\rm ns}$, should also be included in the $\st$ family.
We do not introduce $\sK$ as a separate representative, since such matrices can be symmetrically permuted into the block-tridiagonal form $\sT$. Our choice of $\sT$ as the canonical representative of this subclass of double saddle-point systems reflects both established usage and practical considerations. The block-tridiagonal ordering is the form most commonly adopted in the numerical linear algebra and PDE communities when referring to ``double saddle-point'' systems, and it is the setting in which much of the existing spectral analysis and preconditioning theory has been developed. For these reasons, we adopt $\sT$ as the representative formulation, while noting that matrices of the form \eqref{eq:K} fall within the same class up to symmetric permutation.
\end{remark}

Matrices of the block-arrow form $\sA$ have been studied in, e.g., \cite{bb18,bb19} and arise in applications such as liquid crystal director modeling \cite{rg13}.
Matrices of the block-tridiagonal  form $\sT$ have been analyzed mathematically in, for example, \cite{ak25, bmpp25, bg22, sz18} and arise in applications such as PDE-constrained optimization \cite{bmpp24}.

Having defined the matrix classes of interest, we proceed to survey examples arising from applications in Section \ref{sec:applications}. In Section \ref{sec:class} we describe a framework, derived through a lens of constrained optimization, that unifies the block-arrow and block-tridiagonal forms of double saddle-point matrices. We present structural and spectral properties of these matrices in Section \ref{sec:properties}, which we will then use to motivate and analyze preconditioners (with a focus on Schur complement-based approaches) in Section \ref{sec:iter-prec}. Section \ref{sec:conclusion} presents concluding remarks.

\section{Examples of $\st$ matrices}
\label{sec:applications}
In this section, we present several examples. It is by no means an exhaustive list; we limit ourselves to a few representative instances of the $\st$ class. We note that some of the instances we present do not immediately appear to belong to $\st$, but simple transformations (such as negation or scaling) yield matrices that fit the requirements of Definition \ref{defn:dsp}. Such transformations may not be applied in practice, but the solution methods may be similar. We thus find it useful to include a few such instances.

\subsection{Block-tridiagonal form}
We begin with the family of block-tridiagonal matrices. 
\subsubsection{PDE-constrained optimization}
\label{sec:PDEC}

Consider a discretized optimization problem of the form
\begin{align}
\begin{split}
&\min_{y,u} \frac{1}{2} y^T C y -y^T w + \frac{\beta}{2} u^T R u \\
& \textrm{subject to } Ky + Lu = d,
\end{split}
\label{eq:pdeco}
\end{align}
where $K \in \mathbb{R}^{n \times n}$ is a stiffness matrix corresponding to a partial differential equation; $L \in \mathbb{R}^{n \times m}$ is a control matrix; $C \in \mathbb{R}^{n \times n}$ is a positive semidefinite (and sometimes positive definite) observation matrix; $R \in \mathbb{R}^{m \times m}$ is a positive definite regularization matrix; and $\beta > 0$ is a regularization parameter (often around $10^{-2}$ in practice). 

The Lagrangian of \eqref{eq:pdeco} has the form
\[
\mathcal{L}(y,u,\lambda) = \frac12 y^T C y - y^T w + \frac{\beta}{2} u^T R u + \lambda^T(Ky+Lu-d).
\]

The first-order optimality conditions yield the linear system
\[
\begin{bmatrix}
C & 0 & K^T \\
0 & \beta R & L^T \\
K & L & 0
\end{bmatrix}
\begin{bmatrix}
y \\
u \\
\lambda
\end{bmatrix}
= \begin{bmatrix}
w \\
0 \\
d
\end{bmatrix}.
\]
Reordering the unknowns yields a system of the form $\sT$:
\begin{equation}
\label{eq:pdeco_gen_2sp}
\begin{bmatrix}
C & K^T & 0 \\
K & 0 & L \\
0 & L^T & \beta R
\end{bmatrix}
\begin{bmatrix}
y \\
\lambda \\
u
\end{bmatrix} = \begin{bmatrix}
w \\
d \\
0
\end{bmatrix}.
\end{equation}

Such problems were introduced by Lions \cite{lions1972some}, and solution methods have been investigated  extensively in the literature; see, for example,  \cite{ito1996augmented}, for an early reference.

In many cases, there are additional assumptions or restrictions on the coefficient blocks, and most preconditioners for PDE-constrained optimization are constructed for a more restrictive problem class. For instance, it is common to assume that the operator corresponding to $K$ is elliptic \cite{cfms13,p12}. In this case, the matrix $K$ is singular if the underlying PDE has pure Neumann boundary conditions, and positive definite otherwise. Some papers \cite{cfms13,p12} focus on the case of distributed control and observations, with identical discretizations for the state and control variables. The matrix (with \rv{block-tridiagonal ordering}) then has the form
\[
\begin{bmatrix}
    M & K & 0 \\
    K & 0 & -M \\
    0 & -M & -\beta M
\end{bmatrix}.
\]

In this case, all the block matrices are square and all the nonzero blocks are invertible, which leads to simpler numerical solution approaches. Preconditioners based on the classical saddle-point formulation for these problems have been developed in, for example, \cite{krt21,p12,rdw10,SchoeberlZulehner2007}, while preconditioners based on the double saddle-point formulation have been developed in \cite{bsz20,bg22,pp21,sz18}. 

\subsubsection{Constrained weighted least-squares}
Another example arises from constrained weighted least-squares problems, where we  minimize the residual subject to linear equality constraints (see \cite[Chapter 3]{b24}):
\begin{align*}
\min_y \ \ &||c - Gy ||_2\\
\textrm{s.t.} \ \ &Ey = d.
\end{align*}
First-order optimality conditions for this formulation give rise to the block-tridiagonal system
\[
\begin{bmatrix}
    I & G & 0 \\
    G^T & 0 & E^T \\
    0 & E & 0
\end{bmatrix}\begin{bmatrix}
    r \\
    y \\
    \lambda
\end{bmatrix} = \begin{bmatrix}
    c \\
    0 \\
    d
\end{bmatrix},
\]
where $r$ denotes the residual and $\lambda$ the vector of Lagrange multipliers.

\subsubsection{Dual-dual finite element formulations}
Dual-dual finite element formulations solve second-order elliptic problems by introducing the gradient as an additional unknown (see, e.g., \cite{gh01}). The resulting linear systems exhibit the block structure
\[
\begin{bmatrix}
A & B_1^T & 0 \\
B_1 & 0 & B^T \\
0 & B & 0
\end{bmatrix},
\]
where $A$ is positive definite and the constraint matrices $B_1$ and $B$ typically have full row rank. These methods have been applied in various contexts, including hyperelasticity \cite{gh00}, and have motivated extensions of iterative solvers  such as variants of conjugate gradients and minimum residual methods to handle this structure \cite{GaticaHeuer2002,gh01}.

\subsubsection{Boundary element tearing and interconnecting methods}
Boundary Element Tearing and Interconnecting (BETI) methods are boundary element counterparts of
the well–known Finite Element Tearing and Interconnecting (FETI)
methods \cite{farhat1991method}. In the inexact, data-sparse BETI discretization \cite{losz07}, one obtains the block system
\[
\begin{bmatrix}
V & -K & 0 \\
-K^T & -D & B^T \\
0 & B & 0
\end{bmatrix}
\begin{bmatrix}
t \\
u \\
\lambda
\end{bmatrix} = \begin{bmatrix}
g \\
f \\
0
\end{bmatrix}.
\]
The block-diagonal operators $V$, $D$, and $K$ are defined by elementwise boundary integral operators. The matrix $V$ is positive definite. The matrix $D$ is block diagonal and generally has high nullity, as many of its constituent diagonal blocks are singular with a kernel spanned by the all-ones vector. The coupling matrix $B$ enforces interface continuity of local potentials.

\subsubsection{Poroelasticity}
Poroelasticity describes the coupling between flow in porous media and mechanics. The papers~\cite{Castelletto2016,ferronato2010,ffjct19} provide a comprehensive list of references that discuss related multiphysics models and applications (including the Biot equations) and numerical solution methods based on finite element discretizations and preconditioned iterative solvers.
In \cite{ffjct19}, for example, Ferronato et al.
derive a block Jacobian for the quasi-static poromechanical system of the form
\[
\sK = \begin{bmatrix}
    K & 0 & -Q \\
    0 & A & -B \\
    Q^T & \gamma B^T & P
\end{bmatrix},
\]
where $K$ is the SPD elastic stiffness matrix (and is the largest of the diagonal blocks), $A$ is an SPD mass matrix for Darcy's velocity in mixed form, $P$ is the diagonal capacity matrix for fluid flow, and $Q$ and $B$ are blocks coupling displacements and Darcy's velocities to the pressure unknowns. The parameter \(\gamma>0\) is determined by time discretization, commonly \(\gamma\in\{\Delta t,\,\tfrac12\Delta t\}\).

While $\sK$ does does not belong to the $\st$ class, simple transformations can bring it into this form. (We note, however, that the authors of \cite{ffjct19} do not convert their matrices to symmetric form in their consideration of solvers.) For consistency with our sign convention and to expose a symmetric double saddle‑point structure, we apply \rv{permutation followed by diagonal scaling} (equivalent to redefining the Stokes/flow unknowns by \((\hat{\boldsymbol u},\hat p)=(-\boldsymbol u,\,\gamma p)\)), which yields
\rv{
\begin{eqnarray*}
\tilde{\sK} &=&
\begin{bmatrix}
	I & 0 & 0 \\
	0 & 0 & I \\
	0 & I &  0
	\end{bmatrix}
\sK 
\begin{bmatrix}
	I & 0 & 0 \\
	0 & 0 & I \\
	0 & I &  0
	\end{bmatrix}
\begin{bmatrix}
	I & 0 & 0 \\
	0 & -I & 0 \\
	0 & 0 & \frac{1}{\gamma} I
	\end{bmatrix} \\
    &= & 
    \begin{bmatrix}
K & Q & 0 \\
Q^T & -P & B^T \\
0 & B & \frac{1}{\gamma}A
\end{bmatrix}.
\end{eqnarray*}
}
\subsection{Block arrow}
The second class we consider is the block-arrow form.
\subsubsection{Liquid crystal director modeling} In liquid crystal director modeling, we consider the problem of modeling equilibrium configurations of nematic liquid crystals \cite{rg13}. These systems possess a double saddle‑point structure due to pointwise unit‑vector constraints on the director field $n$ and coupling to an electric field via the electrostatic potential $U$. After discretization and Newton linearization, the inner linear system has the block form
\[
\begin{bmatrix}
    A & B & D \\
    B^T & 0 & 0 \\
    D^T & 0 & -C
\end{bmatrix}\begin{bmatrix}
    \delta n \\
    \delta \lambda \\
    \delta U
\end{bmatrix} = - \begin{bmatrix}
    \nabla_n L \\
    \nabla_{\lambda} L \\
    \nabla_U L
\end{bmatrix},
\]
where $L$ is the discrete Lagrangian. The matrix $A = A_0 + \Lambda$, where $A_0$ represents the distortional stiffness (similar to a discrete vector Laplacian) and $\Lambda$ is a diagonal matrix of Lagrange multipliers. The matrix $B$ encodes the pointwise unit vector constraint on the director field, and $D$ encodes the coupling between $n$ and $U$. The matrix $C$ is a symmetric positive definite matrix arising from the potential equation. We note that for this problem, the matrix $D^T$ is not full rank, but its nullity is 1, and thus satisfies Condition \ref{[C3]} in Definition \ref{defn:dsp}.

Many papers on the block-arrow form of double saddle-point systems have focused on preconditioning for liquid crystal director modeling: see, for example, \cite{bb18,bb18b,ch24,cr23,z25,zl25,zzz25}.

\subsubsection{Interior-point methods}

Consider a quadratic program (QP) in standard form \cite[Chapter 16]{nw06}
\begin{align}
	\begin{split}
	\label{eq:qp_std_form}
	&\min_x \ c^T x + \frac{1}{2}x^T H x \ \ \textrm{ s.t. } \ \ Jx = b, x \ge 0,\\
	&\max_{x, y, z} \ b^T y - \frac{1}{2} x^T H x \ \ \textrm{ s.t. } \ \ J^T y + z -Hx = c, z \ge 0,
	\end{split}
\end{align}
where inequalities are understood elementwise, and $y$ and $z$ are vectors of the Lagrange multipliers. If we solve the quadratic program via a primal-dual interior-point method, the linear system to be solved at each step is given by \cite{gmo14}
\[
\begin{bmatrix}
H & J^{T} & -I \\
J & 0 & 0 \\
-Z & 0 & -X
\end{bmatrix}
\begin{bmatrix}
\Delta x \\[2pt]
-\Delta y \\[2pt]
\Delta z
\end{bmatrix}
=
\begin{bmatrix}
-\,c - Hx + J^{T}y + z \\[2pt]
b - Jx \\[2pt]
XZe - \tau e
\end{bmatrix},
\]
where $X$ and $Z$ are diagonal matrices of the current $x_i$ and $z_i$ iterates, and $\tau$ is a barrier parameter. With added regularization \cite{fo12}, the system has the form
\[
\sA_{ipm} = \begin{bmatrix}
H+\rho I & J^{T} & -I \\
J & -\delta I & 0 \\
-Z & 0 & -X
\end{bmatrix},
\]
where $\rho, \delta > 0$ are regularization parameters. While the system is nonsymmetric, we note that it can easily be symmetrized (see \cite{gmo14} for details) to yield the matrix
\[
\hat{\sA}_{ipm} :=  \begin{bmatrix}
H+\rho I & J^{T} & -Z^{1/2} \\
J & -\delta I & 0 \\
-Z ^{1/2} & 0 & -X
\end{bmatrix}.
\]

\subsubsection{Stokes-Darcy equations}
The comprehensive survey paper of Discacciati \& Quarteroni \cite{discacciati2009navier}
provides a rich description of the Stokes-Darcy equations, covering many aspects, including the physical setup, 
the physical interface conditions (like Beavers–Joseph–Saffman), and numerical solution techniques. Following the formulation of \cite{cmx09,cls16}, the discrete Stokes--Darcy interface problem leads to the $3\times3$ block system
\begin{equation}
\label{eqn:darcystokes}
\begin{bmatrix}
A_p & A_{\Gamma}^{T} & 0 \\
- A_{\Gamma} & A_f & B_f^{T} \\
0 & B_f & 0
\end{bmatrix}
\begin{bmatrix}
\phi_h \\[2pt]
\mathbf{u}_h \\[2pt]
p_h
\end{bmatrix}
=
\begin{bmatrix}
f_{p,h} \\[2pt]
\mathbf{f}_{f,h} \\[2pt]
g_h
\end{bmatrix}.
\end{equation}
Here $\phi_h$ denotes the piezometric head (a scaled/shifted Darcy pressure), while $\mathbf{u}_h$ and $p_h$ are the Stokes velocity and pressure. In the block matrix, $A_p$ is the Darcy stiffness (SPD), $A_f$ is the vector Laplacian (SPD), $A_{\Gamma}$ encodes the interface coupling conditions, and $B_f$ is the discrete divergence.

Equation \eqref{eqn:darcystokes} does not immediately fit into our definition of the $\st$ class because both nonzero diagonal blocks are positive definite and the (2,2)-block is larger than the (1,1)-block. However, after symmetric permutation and sign adjustment, we can obtain a double saddle-point system as follows:
\begin{equation}
\begin{bmatrix}
    A_f & -A_{\Gamma} & B_f^T \\
    -A_{\Gamma}^T & -A_p & 0 \\
    B_f^T & 0 & 0
\end{bmatrix}\begin{bmatrix}
    \mathbf{u}_h \\
    \phi_h \\
    p_h
\end{bmatrix} = \begin{bmatrix}
    \mathbf{f}_{f,h} \\
    -f_{p,h} \\
    g_h
\end{bmatrix}.
\label{eq:SDsymm}
\end{equation}
We note that the Marker-and-Cell formulation of Stokes-Darcy described in \cite{g26, gh23} yields a similar block structure to the one in \cite{cmx09,cls16}, but with the difference that the off-diagonals are written in symmetric form, while the middle diagonal block is (mildly) nonsymmetric due to the choice of the discretization. Therefore, that formulation does not belong to the $\st$ class even after block permutation.

\subsubsection{Magma/mantle dynamics} 
The two-phase flow equations that describe coupled magma/mantle dynamics
were derived by McKenzie \cite{mckenzie1984generation}.
Rhebergen et al. \cite{rwwk15} consider these equations on a domain  $\Omega \subset \mathbb{R}^d$, where $1 \le d \le 3$. The governing equations for coupled magma/mantle flow introduce a \emph{compaction pressure} $p_c$ to avoid an explicit grad--div term and yield a three-unknown system $(u,p,p_c)$ on $\Omega\subset\mathbb{R}^d$:
\begin{align*}
-\nabla \cdot \Big(\eta\big(Du-\tfrac{1}{3}(\nabla\!\cdot u)I\big)\Big) + \nabla p + \nabla p_c &= \phi\,e_3,\\
-\nabla \cdot u + \nabla \cdot \big(k \nabla p\big) &= \nabla \cdot \big(k e_3\big),\\
-\nabla \cdot u - \zeta^{-1}p_c &= 0,
\end{align*}
with shear viscosity $\eta>0$, bulk viscosity $\zeta>0$, permeability $k\ge 0$, porosity $\phi\in[0,1]$, and $e_3$ the unit vector aligned with gravity. Standard Dirichlet/Neumann data close the system.

A mixed finite element discretization (using an inf--sup stable velocity/pressure pair) leads to the symmetric $3\times3$ block system
\[
\begin{bmatrix}
\eta K & G^{T} & G^{T}\\
G & -kC & 0 \\
G & 0 & -\zeta^{-1}Q
\end{bmatrix}
\begin{bmatrix}
u\\[2pt] p\\[2pt] p_c
\end{bmatrix}
=
\begin{bmatrix}
f\\ g\\ 0
\end{bmatrix},
\]
where $K$ is the (SPD) elasticity operator including the grad--div contribution, $G$ is the discrete divergence (so $G^{T}$ is the negative discrete gradient), $C$ is the permeability-weighted pressure Laplacian, and $Q$ is the pressure mass matrix. Uniqueness of $p$ is enforced by fixing its mean so that $\ker(G^{T})=\operatorname{span}\{1\}$.

\subsubsection{Flow in fractured porous media}

Fractured porous media \cite{WarrenRoot1963}
describes materials like rocks or soils that contain both a porous matrix and networks of fractures, where fluid flow and transport occur through the combined effects of pores and cracks.
The paper of Antonietti et al. \cite{apfs20} targets the iterative solution of linear systems arising from hybrid‑dimensional Darcy flow in fractured porous media, where fractures are modeled as $(d-1)$-dimensional interfaces coupled to a $d$-dimensional porous matrix. The resulting algebraic system is a double saddle‑point problem for unknowns velocity $u$, pressure $p$, and fracture pressure $p_{\Gamma}$:
\[
\begin{bmatrix}
    M_c & B^T & C^T \\
    B & 0 & 0 \\
    C & 0 & -T
\end{bmatrix}\begin{bmatrix}
    u \\
    p \\
    p_{\Gamma} 
\end{bmatrix}= \begin{bmatrix}
        g \\
        h \\
        h_{\Gamma}
    \end{bmatrix}.
\]
Here $M_c$ is an SPD inner-product matrix including coupling terms on the fracture interface, $B$ is a discrete divergence operator, $C$ couples normal flux with fracture pressure, and $T$ is the positive semidefinite transmissibility matrix of the fracture discretization.

\section{Optimization interpretation}
\label{sec:class}

A large number of double saddle-point systems encountered in practice originate from the numerical solution of partial differential equations, which frequently involve constraints such as incompressibility. The resulting block structure can often be interpreted within the framework of constrained optimization. In this section, we introduce a constrained optimization perspective that accommodates both block-arrow and block-tridiagonal formulations.

The purpose of this section is to give a
unifying optimization interpretation for the various matrix structures that
appear under the label ``double saddle-point.'' In particular, we show that both
block-arrow and block-tridiagonal formulations arise naturally as first-order optimality (KKT) conditions of closely related constrained quadratic programs. This perspective also clarifies the connection between double saddle-point systems and the broader class of multiple saddle-point problems. By making these relationships explicit, we reconcile the multiple uses of the term ``double saddle-point'' found in the literature, and also clarify why the term ``saddle-point'' is appropriate in the description of all the matrices under discussion in this paper. 

That said, we recognize that the optimization formulation that we offer in this section is tailored to the setting of $\st$ matrices in terms of the requirement of symmetry and (semi)definiteness of the blocks. It would be harder to provide such interpretations for nonsymmetric formulations that often arise in multiphysics applications but are not covered under the $\st$ family. We also emphasize that the optimization viewpoint developed in this section is intended as a structural interpretation, not as a claim about the origin of all matrices in the $\st$ family. As seen in Section \ref{sec:applications}, while many double saddle-point systems do indeed arise as first-order optimality conditions of constrained optimization problems, many others (particularly those originating from the discretization of partial differential equations) do not. Nevertheless, the algebraic structure of these systems admits an interpretation that is formally equivalent to such optimality conditions, and it is this viewpoint that we exploit to unify the block-arrow and block-tridiagonal formulations and to connect double saddle-point systems to both classical saddle-point systems and to their multiple saddle-point generalizations.

\subsection{A standard quadratic program with equality constraints}
Consider the quadratic programming problem with equality constraints \cite{nw06}
\begin{align*}
\text{Minimize} \quad & \frac{1}{2} x^T H x - c^T x \\
\text{Subject to} \quad & J x = d,
\end{align*}
where:
 \( x \in \mathbb{R}^n \) contains the (primal) variables,
   \( H \in \mathbb{R}^{n \times n} \) is a symmetric (typically positive semidefinite) Hessian, 
   \( J \in \mathbb{R}^{m \times n} \) is the Jacobian matrix of constraints, and
  \( c \in \mathbb{R}^n \), \( d \in \mathbb{R}^m \) are given vectors.

The Lagrangian for this problem is given by
\[
\mathcal{L}(x, \lambda) = \frac{1}{2} x^T H x - c^T x + \lambda^T (J x - d).
\]
Here, the vector $\lambda \in \mathbb{R}^m   $ contains the Lagrange multipliers. 
To solve the problem, we set the gradient of the Lagrangian with respect to \( x \) and \( \lambda \) to zero:
\begin{align*}
\nabla_x \mathcal{L}(x, \lambda) &= H x - c + J^T \lambda = 0; \\
\nabla_\lambda \mathcal{L}(x, \lambda) &= J x - d = 0.
\end{align*}
Combining these conditions yields the classical saddle-point system
\begin{equation}
    \label{eq:classical_sp}
\begin{bmatrix}
H & J^T \\
J & 0
\end{bmatrix}
\begin{bmatrix}
x \\
\lambda
\end{bmatrix}
=
\begin{bmatrix}
c \\
d
\end{bmatrix}.
\end{equation}

Here the term ``saddle-point'' is used because the solution of the linear system is a saddle-point of the Lagrangian: the stationary point is a minimizer with respect to the primal variables and a maximizer with respect to the Lagrange multipliers. This is our basic building block for constructing double saddle-point systems.
\subsection{Constrained optimization yielding double saddle-point systems}
In the classical saddle-point system just discussed, the constrained optimization framework involves one set of primal variables and one set of constraints. In contrast, the $3 \times 3$-block case provides two distinct options. It could consist of 
one set of primal variables and two sets of constraints, resulting in the block-arrow structure. Or it could consist of
two sets of primal variables and one set of constraints, yielding the block-tridiagonal structure.

\subsubsection*{Block-arrow double saddle-point system}
For the block-arrow case, we consider the constrained optimization problem
\begin{align*}
\text{Minimize} \quad & \frac{1}{2} x_1^T H_1 x_1  - c_1^T x_1  \\
\text{subject to} \quad & J_1 x_1 = d_1 {\quad \rm and\quad }J_2 x_1 = d_2.
\end{align*}

We define the Lagrangian as
$$ \mathcal{L}(x_1,\lambda_1, \lambda_2) = \frac{1}{2} x_1^T H_1 x_1    -c_1^T x_1  + \lambda_1^T (J_1x_1 -d_1)+ \lambda_2^T (J_2 x_1 -d_2),$$
where $\lambda_1$ and $\lambda_2$ are the two vectors of Lagrange multipliers.
We differentiate and equate the gradient of $\mathcal{L}(x_1,\lambda_1,\lambda_2)$ to zero, and this gives us
\begin{equation}
\label{eq:opt22}
\begin{bmatrix}
H_1 & J_1^T  & J_2^T  \\
J_1 & 0 & 0 \\
J_2 & 0 & 0
\end{bmatrix}\begin{bmatrix}
    x_1 \\
    \lambda_1 \\
    \lambda_2
\end{bmatrix} = \begin{bmatrix}
    c_1 \\
    d_1 \\
    d_2
\end{bmatrix}.
\end{equation}
The matrix of \eqref{eq:opt22}  is a simple version of $\mathcal{A}$ defined in \eqref{eq:2sp},  where $A_1=H_1$, $A_2=A_3=0$, and $B_i=J_i, i=1,2$.

\subsubsection*{Block-tridiagonal double saddle-point system}
Consider the constrained optimization problem
\begin{align*}
\text{Minimize} \quad & \frac{1}{2} x_1^T H_1 x_1 + \frac{1}{2} x_2^T H_2 x_2 - c_1^T x_1 - c_2^T x_2 \\
\text{subject to} \quad & J_1 x_1 +J_2 x_2 = d_1.
\end{align*}
We define the Lagrangian as
$$ \mathcal{L}(x_1,x_2,\lambda) = \frac{1}{2} x_1^T H_1 x_1 + \frac12 x_2^T H_2 x_2  -c_1^T x_1 - c_2^T x_2 + \lambda^T (J_1x_1 +J_2x_2 -d_1),$$
where $\lambda$ is the vector of Lagrange multipliers.
We differentiate and equate the gradient of $\mathcal{L}(x_1,x_2,\lambda)$ to zero, and this gives us
\begin{equation}
\label{eq:opt}
\begin{bmatrix}
H_1 & 0 & J_1^T  \\
0 & H_2 & J_2^T \\
J_1 & J_2 & 0
\end{bmatrix}\begin{bmatrix}
    x_1 \\
    x_2 \\
    \lambda
\end{bmatrix} = \begin{bmatrix}
    c_1 \\
    c_2 \\
    d_1
\end{bmatrix}.
\end{equation}
After symmetrically permuting the second and third block-rows and block-columns, we obtain a specific instance of $\mathcal{T}$ defined in \eqref{eq:2sp}, where $A_1=H_1$, $A_2=0, A_3=H_2,$ $B_1=J_1,$ and $B_2=J_2^T$; the transposition in the last equality follows from the symmetric permutation.


\subsubsection*{Regularized extensions}
For the block-arrow form, we can add two quadratic penalty terms (corresponding to the two constraints) to the original optimization problem, resulting in
\begin{align*}
\text{Minimize} \quad & \frac{1}{2} x_1^T H_1 x_1 - c_1^T x_1 + \frac{1}{2} \lambda_1^T R_1 \lambda_1 + \frac{1}{2} \lambda_2^T R_2 \lambda_2 \\
\text{subject to} \quad & J_1 x_1 = d_1 \quad \text{and} \quad J_2 x_1 = d_2,
\end{align*}
where $R_1$ and $R_2$ are positive semidefinite. Setting the gradient of the Lagrangian to zero and reordering generates the linear system
\[
\begin{bmatrix}
H_1 & J_1^T  & J_2^T  \\
J_1 & -R_1 & 0 \\
J_2 & 0 & -R_2
\end{bmatrix}\begin{bmatrix}
    x_1 \\
    \lambda_1 \\
    \lambda_2
\end{bmatrix} = \begin{bmatrix}
    c_1 \\
    d_1 \\
    d_2
\end{bmatrix}.
\]

For the block-tridiagonal formulation, consider the original constrained optimization problem with one added quadratic penalty term
\begin{align*}
\text{Minimize} \quad & \frac{1}{2} x_1^T H_1 x_1 + \frac{1}{2} x_2^T H_2 x_2 - c_1^T x_1 - c_2^T x_2 + \frac{1}{2} \lambda^T R_1 \lambda \\
\text{subject to} \quad & J_1 x_1 + J_2 x_2 = d_1.
\end{align*}
Setting the gradient of the Lagrangian to zero and reordering generates the linear system
\[
\begin{bmatrix}
H_1 & 0 & J_1^T  \\
0 & H_2 & J_2^T \\
J_1 & J_2 & -R_1
\end{bmatrix}\begin{bmatrix}
    x_1 \\
    x_2 \\
    \lambda
\end{bmatrix} = \begin{bmatrix}
    c_1 \\
    c_2 \\
    d_1
\end{bmatrix},
\]
which, upon reordering the second and third blocks, yields the block-tridiagonal form.

\subsection{Multiple saddle-point systems}
\label{sec:multiple}
Multiple saddle-point systems are extensions of the block-tridiagonal double saddle-point system $\sT$. A triple (or threefold) saddle-point system is defined by
\begin{equation}
\label{eq:3sp}
\sT_4 = \begin{bmatrix}
H_1 &  J_1^T &  0 & 0  \\
J_1 & -R_1 & J_2^T & 0 \\
0 & J_2 & H_2 & J_3^T \\
 0 & 0 &  J_3 & -R_2
\end{bmatrix}.
\end{equation}
Similarly to our descriptions of the block-tridiagonal and block-arrow cases in the previous subsection, we note that we can write this as a constrained optimization problem with two primal and two constraint block variables
\begin{equation}
\label{eq:3sp_our_ordering}
\sK_4 = \left[\begin{array}{c c | c c}
H_1 & 0 & J_1^T & 0  \\
0 & H_2 & J_2 & J_3^T \\
\hline
J_1 & J_2^T & -R_1 & 0 \\
0& J_3 & 0 & -R_2
\end{array}\right], 
\end{equation}
where we have inserted a horizontal and a vertical line to highlight the partitioning that gives rise to a classical saddle-point matrix. Note that the $J_i$ matrices are  transposed in an alternating fashion in \eqref{eq:3sp_our_ordering}, for consistency with \eqref{eq:3sp}.

The matrix $\sT_4$ can be reordered to obtain the triple saddle-point matrix $\sK_4$ by symmetrically permuting the 2- and 3-blocks.

A number of papers \cite{bb26,bmpp25,pp21,sz18}  analyze \textit{multiple} saddle-point systems. These are block-$n \times n$, block-tridiagonal generalizations of classical and double saddle-point systems that take the form
\begin{subequations}
\label{eq:block-tridiag}
\begin{align}
\sT_n &\;=\;
\begin{bmatrix}
H_{1}      & J_{1}^T   &             &             &            &\\
J_{1}      & -R_{1}   & J_{2}^T      &             &            &\\
           & J_{2}    & H_{2}       & \ddots      &            &\\
           &          & \ddots      & \ddots      & J_{n-2}^T  & \\
           &          &             & J_{n-2}     & H_{n/2} & \rv{J_{n-1}^T} \\
           &          &              &            & J_{n-1} & -R_{n/2}
\end{bmatrix} \qquad \textrm{for even $n$}, \\
\textrm {or} \nonumber \\
\sT_n &\;=\;
\begin{bmatrix}
H_{1}      & J_{1}^T   &             &             &            &\\
J_{1}      & -R_{1}   & J_{2}^T      &             &            &\\
           & J_{2}    & H_{2}       & \ddots      &            &\\
           &          & \ddots      & \ddots      & J_{n-2}^T  & \\
           &          &             & J_{n-2}     & -R_{(n-1)/2} & \rv{J_{n-1}^T} \\
           &          &              &            & J_{n-1} & H_{(n+1)/2}
\end{bmatrix} \qquad \textrm{for odd $n$.}
\end{align}
\end{subequations}

In accordance  with our notation for this section, we use \( H_i \) to denote the  primal block variables and \( R_i \) to denote the regularizing terms. Extending our illustration of the triple saddle-point case, we observe that multiple saddle-point systems can  be framed within the constrained optimization setup introduced in this section. Specifically, such systems can be represented as a constrained optimization problem involving \( k \) primal block variables and \( k \) or \( k-1 \) block constraints, where $k = \lceil n/2  \rceil$. That is, if the matrix is written in the classical saddle-point form \eqref{eq:classical_sp}, then $H={\rm blkdiag} \{H_1, H_2,\dots \}$, and the constraint matrix $J$ is a block upper bidiagonal matrix; see Figure~\ref{fig:rb_permut} for further illustration. After applying red-black reordering \cite[p. 159]{young1971iterative} using the permutation vector
\[
\sigma = (1, k+1, 2, k+2, \ldots),
\]
we obtain the block-tridiagonal multiple saddle-point matrix  in~\eqref{eq:block-tridiag}.

 \begin{figure}[tbh!]
 \begin{center}
\includegraphics [width=0.7\textwidth]{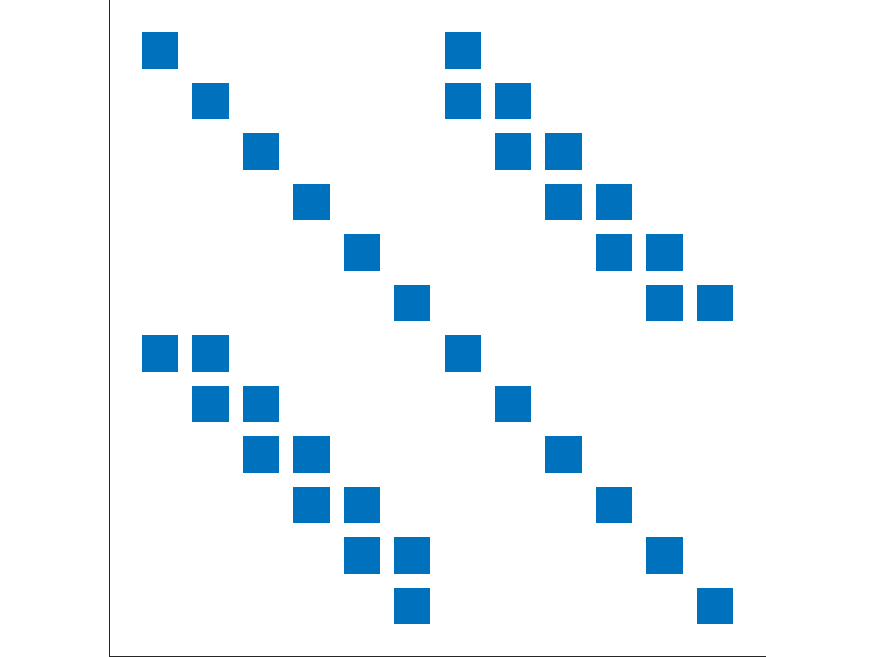}
\caption{An illustration of the permutation mechanism described in Section \ref{sec:multiple}: plot of the nonzero block structure of a block-$12 \times 12$ multiple saddle-point system, before permutation to a block-tridiagonal form.\label{fig:rb_permut} } 
\end{center}
\end{figure}

\section{Structural and Spectral Properties of $\st$
Matrices}
\label{sec:properties}

With the family $\st$ of double saddle-point matrices defined, a number of algebraic and spectral properties emerge from the underlying block structure. In this section we explore block-LDL factorizations, invertibility conditions, and spectral properties of these systems. These results unify and extend observations made in the literature and encompass both the block-tridiagonal and block-arrow formulations. Beyond their intrinsic interest, such spectral properties form the basis of various numerical solvers used in practice. One such widely used approach -- namely, Schur complement-based block preconditioning -- is described in Section \ref{sec:iter-prec}.

While $A_1$ is required to be nonzero, it is common to have applications where one or both of $A_2$ and $A_3$ are zero. This leads to additional necessary conditions for invertibility, as follows.
\begin{lemma}[Dimension restrictions for invertibility]   
The following dimension restrictions for invertibility of matrices in the $\st$ family hold:
\begin{enumerate}[label=(\roman*)]
\item For $\sA$, if $A_2=A_3=0$ then a necessary condition for invertibility is  $n_1 \ge n_2 + n_3$.
\item For $\sT$, if $A_3=0$ then a necessary condition for invertibility is $n_1 \ge n_2 \ge n_3$.  
\end{enumerate}
\label{lem:dimr}
\end{lemma}
\begin{proof}
For $\sA$, the assertion follows immediately by properties of (classical) saddle-point systems; if $A_2=A_3=0$ we can refer to $\sA$ as a saddle-point matrix whose leading block is $A_1$, and hence a necessary condition is $n_1 \ge n_2+n_3$.

For $\sT$, consider the case where $A_3=0$. In this case, then $B_2$ must have full row rank and no more rows than columns. Thus, we require  $n_2 \ge n_3$. Since we have $n_1 \ge \max \{ n_2,n_3 \} $ by \ref{[C1]}, we end up with $n_1 \ge n_2 \ge n_3.$
\end{proof}

\subsection{Block-LDL decompositions} 
\label{sec:blockLDL}
The $\st$ matrices are symmetric indefinite, and hence they admit a block-LDL decomposition, which is useful not only in its own right but also as a building block for the design of block preconditioners. In this section we present those block decompositions.

Throughout this section, for simplicity and clarity, we will assume that  \ref{[C1]}--\ref{[C3]} hold with the stronger condition that $B_1$ and $B_2$ have full rank, and that the dimension restrictions stated in Lemma \ref{lem:dimr} hold.

\subsubsection{Block-arrow LDL decomposition}
\label{sec:block-arrow-ldl}
For the block-arrow matrix, we have
\[
\sA = \sL_a \sD_a \sL_a^T, 
\]
where
\begin{align}
\label{eq:arrow_ldl}
        \sL_a &=
		\begin{bmatrix}
			I & 0 & 0 \\
B_1 A_1^{-1} & I & 0 \\
B_2 A_1^{-1} & B_2 A_1^{-1} B_1^T S_{a,1}^{-1} & I
		\end{bmatrix} \quad {\rm and} \quad \sD_a = \begin{bmatrix}
				A_1 & 0 & 0 \\
				0 & -S_{a,1} & 0 \\
				0 & 0 & -S_{a,2}
		\end{bmatrix}.
\end{align}
The Schur complements $S_{a,1}$ and $S_{a,2}$ are given by
\begin{subequations}
    \label{eq:SCa}
    \begin{gather}
S_{a,1} = A_2 + B_1 A_1^{-1} B_1^T, 
\label{eq:sca1} \\
S_{a,2} = A_3 + B_2 A_1^{-1} B_2^T - B_2 A_1^{-1} B_1^T S_{a,1}^{-1} B_1 A_1^{-1} B_2^T.
 \label{eq:sca2}
\end{gather}
\end{subequations}

\begin{proposition}[Definiteness of block-arrow Schur complements] \rv{The matrix $S_{a,1}$ is} symmetric positive definite. The matrix $S_{a,2}$ is symmetric positive semidefinite, and is positive definite if and only if $\sA$ is nonsingular.
\label{prop:spd}
\end{proposition}
\begin{proof}
    \rv{The positive definiteness of $S_{a,1}$ follows immediately from the stated assumptions that $A_1$ is positive definite and $B_1$ has full row rank.} It remains  to show that $S_{a,2}$ is symmetric positive semidefinite, and to characterize when it is positive definite.

Since $A_1$ is  symmetric positive definite, its square root $A_1^{1/2}$ exists and is invertible. Let us introduce the symmetric matrices
\begin{equation}
Q = A_1^{-1/2} B_1^\T S_{a,1}^{-1} B_1 A_1^{-1/2} \quad\text{and}\quad W = I - Q.
\end{equation}
Then
\begin{equation*}
S_{a,2} = A_3 + B_2 A_1^{-1/2} W A_1^{-1/2} B_2^\T.
\end{equation*}
We now  show that $W$  is symmetric positive semidefinite.  
Let $$P=A_1^{-1/2} B_1^T (B_1 A_1^{-1} B_1^T)^{-1} B_1 A_1^{-1/2}. $$
The matrix $S_{a,1}$ is symmetric positive definite and $S_{a,1}^{-1} \preceq (B_1 A_1^{-1} B_1^\T)^{-1}$.
We perform left and right multiplication by $A_1^{-1/2} B_1^\T$ and obtain $Q \preceq P$. The matrix $P$ is the orthogonal projector onto $\range(A_1^{-1/2} B_1^\T)$, so $0\preceq P\preceq I$. Therefore $Q\preceq I$ and $W=I-Q\spsd$.

From this it follows that $S_{a,2}$ is positive semidefinite. For invertible $\sA$, $\mathcal{D}_a$ has no zero eigenvalues (by Sylvester's Law of Inertia); it follows that $S_{a,2}$ is positive definite in this case.
\end{proof}

\subsubsection{Block-tridiagonal LDL decomposition}
\label{sec:block-tri-ldl}
For $\sT$, we can write \cite{bg22}
\begin{equation}
		\label{eq:tri_ldl}
		\hspace{-2mm}
		\sT =
        \underbrace{
		\begin{bmatrix}
			I & 0 & 0 \\
			B_1A_1^{-1} & I & 0 \\
			0 & -B_2S_{t,1}^{-1} & I
		\end{bmatrix}}_{=: \sL_t}
		\underbrace{
			\begin{bmatrix}
				A_1 & 0 & 0 \\
				0 & -S_{t,1} & 0 \\
				0 & 0 & S_{t,2}
		\end{bmatrix}}_{=: \sD_t}
        \underbrace{
		\begin{bmatrix}
			I & A_1^{-1}B_1^T & 0 \\
			0 & I & -S_{t,1}^{-1}B_2^T \\
			0 & 0 & I
		\end{bmatrix}}_{=\sL_t^T}.
	\end{equation}
The Schur complements $S_{t,1}$ and $S_{t,2}$ are given by
\begin{subequations}
\begin{gather}
S_{t,1} = A_2 + B_1 A_1^{-1} B_1^T, \\
S_{t,2} = A_3 + B_2S_{t,1}^{-1}B_2^T.
\end{gather}
\end{subequations}
It readily follows from Conditions \ref{[C1]}--\ref{[C3]} that $S_{t,1}$ and $S_{t,2}$ are symmetric positive semidefinite, and are symmetric positive definite if $\sT$ is invertible.

We remark that for tridiagonal multiple saddle-point systems $\sT_n$ (see Equation \eqref{eq:block-tridiag}), we can recursively define the Schur complements
\rv{\[
S_{t,i} = A_{i+1} + B_i S_{t,i-1}^{-1} B_i^T, \qquad S_{t,0}=A_1,
\]}%
and the block-lower bidiagonal $L$ factor can similarly be defined easily for large numbers of blocks. These factorizations lead to preconditioners with special properties and eigenvalue bounds; we refer to \cite{pp21,sz18} for details.

\subsubsection{Block-LDL decompositions based on classical partitioning}
\label{sec:classical}

Recall that matrices in $\st$ can be partitioned (after reordering, in the case of $\sT$) into classical saddle-point systems. As such, some results -- including block-LDL factorizations -- can be obtained for the double saddle-point case by using classical saddle-point methods on the partitioned systems. We recall the formula for block-LDL factorization of a classical saddle-point system \cite[Eq. (3.1)]{bgl05}:
\begin{equation}
\label{eq:classical_ldl}
\begin{bmatrix}
    A & B^T \\
    B & -C
\end{bmatrix} = \begin{bmatrix}
    I & 0 \\
    BA^{-1} & I
\end{bmatrix} \begin{bmatrix}
    A & 0 \\
    0 & -S
\end{bmatrix}\begin{bmatrix}
    I & A^{-1}B^T \\
    0& I
\end{bmatrix},
\end{equation}
where $S = C+ BA^{-1}B^T$.

In \cite{BeikBenzi2018}, the authors consider the following two partitionings for the block-arrow form:
\begin{equation}
\label{eq:partitioning}
\sA =
\left[ \begin{array}{c | c  c}
    A_1 & B_1^T & B_2^T \\
    \hline
    B_1 & -A_2 & 0 \\
    B_2 & 0 & -A_3
    \end{array}
\right] =
\left[ \begin{array}{c  c | c}
    A_1 & B_1^T & B_2^T \\
    B_1 & -A_2 & 0 \\
    \hline
    B_2 & 0 & -A_3
    \end{array}
\right].
\end{equation}
In this case, one can obtain  block-LDL decompositions using the classical saddle-point LDL factorization \eqref{eq:classical_ldl}, and these decompositions will be different from \eqref{eq:arrow_ldl}. For the first partitioning in \eqref{eq:partitioning} we get
\[
\sA = \begin{bmatrix}
    I & 0 & 0 \\
    B_1 A_1^{-1} & I & 0 \\
    B_2 A_1^{-1} & 0 & I
\end{bmatrix} \begin{bmatrix}
    I & 0 & 0 \\
    0 & -S_1 & -B_1A_1^{-1}B_2^T \\
    0 & -B_2 A_1^{-1}B_1^T & -S_2
\end{bmatrix}\begin{bmatrix}
    I & A_1^{-1}B_1^T & A_1^{-1}B_2^T \\
    0 & I & 0 \\
    0 & 0 & I
\end{bmatrix},
\]
where $S_1 = A_2 + B_1A_1^{-1}B_1^T$ and $S_2= A_3 + B_2A_1^{-1}B_2^T.$ For the second partitioning in \eqref{eq:partitioning} we get

\[
\sA = \begin{bmatrix}
    I & 0 & 0 \\
    0 & I & 0 \\
    L_{31} & L_{32} & I
\end{bmatrix} \begin{bmatrix}
    A_1 & B_1^T & 0 \\
    B_1 & -A_2 & 0 \\
    0 & 0 & -S_2
\end{bmatrix} \rv{\begin{bmatrix}
    I & 0 & L_{31}^T \\
    0 & I & L_{32}^T \\
    0 & 0 & I
\end{bmatrix}},
\]
where
\begin{align*}
    S_1 &= A_2 + B_1A_1^{-1}B_1^T; \qquad
    S_2 = A_3+B_2 (A_1^{-1} -A_1^{-1} B_1^T S_1^{-1} B_1 A_1^{-1})B_2^T; \\
    L_{31} &= B_2 A_1^{-1} - B_2A_1^{-1}B_1^T S_1^{-1}B_1A_1^{-1}; \qquad
    L_{32} = B_2 A_1^{-1}B_1^T S_1^{-1}.
\end{align*}

Similarly, we can partition $\sK$ (the block-permuted variant of $\sT$ defined in \eqref{eq:K}) as: 
\begin{equation}
\label{eq:partitioning_tri}
 \sK = \left[ \begin{array}{c c | c}
    A_1 & 0 & B_1^T \\
    0 & A_2 & B_2^T \\
    \hline
    B_1 & B_2 & -A_3
    \end{array}
\right] = \left[ \begin{array}{c | c c}
    A_1 & 0 & B_1^T \\
    \hline
    0 & A_2 & B_2^T \\
    B_1 & B_2 & -A_3
    \end{array}
\right].
\end{equation}
The first partitioning in \eqref{eq:partitioning_tri} gives
\[
\sK = \rv{\begin{bmatrix}
    I & 0 & 0 \\
    0 & I & 0 \\
    B_1 A_1^{-1} & B_2 A_2^{-1} & I
\end{bmatrix} }\begin{bmatrix}
    A_1 & 0 & 0 \\
    0 & A_2 & 0 \\
    0 & 0 & -S
\end{bmatrix} \rv{\begin{bmatrix}
    I & 0 & A_1^{-1}B_1^T \\
    0 & I & A_2^{-1}B_2^T \\
    0 & 0 & I
\end{bmatrix}},
\]
where
\rv{\[
S = A_3 + B_1A_1^{-1}B_1^T + B_2 A_2^{-1}B_2^T.
\]}

The second partitioning gives
\[
\sK = \begin{bmatrix}
    I & 0 & 0 \\
    0 & I & 0 \\
    B_1A_1^{-1} & 0 & I
\end{bmatrix} \begin{bmatrix}
    A_1 & 0 & 0 \\
    0 & A_2 & B_2^T \\
    0 & B_2 & -S
\end{bmatrix}\begin{bmatrix}
    I & 0 & A_1^{-1}B_1^T \\
    0 & I & 0 \\
    0 & 0 & I
\end{bmatrix},
\]
where $S = A_3 + B_1 A_1^{-1} B_1^T$.
We note that these Schur complements are different than the ones we obtained in Sections \ref{sec:block-arrow-ldl}--\ref{sec:block-tri-ldl}.  The question of which factorization is more practical for the development of numerical solvers depends on the problem at hand.

\subsection{Invertibility conditions}
\label{sec:inv}

In this section we state necessary and sufficient conditions for invertibility of matrices in $\st$. 

\subsubsection{Block-arrow invertibility conditions}

We begin with a statement of some necessary conditions for invertibility of $\sA$.

\begin{theorem}[$\sA$, necessary conditions for invertibility]
    If $\sA$ is invertible, the following must hold:

    \begin{enumerate}[label=(\roman*)]
        \item $\ker(A_1) \cap \ker(B_1) \cap \ker(B_2) = \{ 0 \}$;
        \item $\ker(B_1^T) \cap \ker(A_2) = \{ 0 \}$;
        \item $\ker(B_2^T) \cap \ker(A_3) = \{ 0 \}$.
    \end{enumerate}
\end{theorem}

\begin{proof}
    To prove statement (i), assume to the contrary that  the intersection of kernels is non-empty and there exists a nonzero vector $x$ satisfying $A_1 x = B_1 x = B_2 x = 0$. It would then be the case that the block column vector $[x, 0, 0]$ is a nonzero null vector of $\sA$, which would imply $\sA$ is singular and is a contradiction. Similar reasoning proves (ii) and (iii).
\end{proof}

In general, by \eqref{eq:arrow_ldl}, a sufficient condition for invertibility when Conditions \ref{[C1]}--\ref{[C3]} hold is that $A_1$, $S_{a,1}$ and $S_{a,2}$ are positive definite. If $B_1$ and $B_2$ have full row rank, we can state  alternative sufficient conditions.

\begin{theorem}[$\sA$, sufficient conditions for invertibility]
\label{thm:invertibility-suff-arrow}
Assume $A_1\spd$, $A_2\spsd$, $A_3\spsd$, and $B_1,B_2$ have full row rank. If either
\begin{enumerate}[label={(\roman*)}]
  \item \label{item:A3SPD} $A_3\spd$ \\ 
   or
  \item \label{item:kernels-disjoint} $\range\big(B_1^T\big) \cap \range\big( B_2^T\big) = \{0\}$,
\end{enumerate}
then  $\sA$ is invertible. 
\label{thm:sufficient}
\end{theorem}

\begin{proof}
Given the positive definiteness of $A_1$ and $S_{a,1},$ invertibility of $\sA$ boils down to invertibility of $S_{a,2}$, per \eqref{eq:arrow_ldl}. Because $S_{a,2} = A_3 + B_2 A_1^{-1/2}WA_1^{-1/2}B_2^T$, where $W$ is as defined in the proof of Proposition \ref{prop:spd} and was shown there to be positive semidefinite. We thus have $S_{a,2} \succeq A_3$. Therefore, if $A_3$ is SPD then so is $S_{a,2}$, which establishes condition \ref{item:A3SPD} as a sufficient condition for invertibility of $\sA$.

If condition \ref{item:kernels-disjoint} is satisfied, the row spaces of $B_1$ and $B_2$ are non-intersecting, and thus the rows of the matrix $\begin{bmatrix} B_1 \\ B_2 \end{bmatrix}$ are linearly independent, implying that the has full row rank, by the requirement that $n_2 + n_3 \le n_1$ (Lemma \ref{lem:dimr}). Therefore, we can use the classical saddle-point partitioning of $\sA$ with $A_1$ as its leading block (see \eqref{eq:2sp}) and conclude that $\sA$ is invertible, by \cite[Theorem 3.1]{bgl05}.
\end{proof}
\rv{
\begin{remark} While conditions (i) and (ii) in Theorem \ref{thm:sufficient} will establish the invertibility of $\sA$, they may not arise in applications of interest; we emphasize that these are sufficient, rather than necessary, conditions. For example, consider the Stokes-Darcy problem described in \eqref{eq:SDsymm}: we have $A_3=0$ and $B_1$ is rank deficient, but $\sA$ is nonetheless invertible in general. 
\end{remark}
}

\subsubsection{Block-tridiagonal invertibility conditions}

For necessary conditions for invertibility of $\sT$, we recall \cite[Proposition 2.1]{bg22}:

\begin{theorem}[$\sT$, necessary conditions for invertibility]
    If $\sT$ is invertible, the following must hold:
    \begin{enumerate}[label=(\roman*)]
        \item $\ker(A_1) \cap \ker(B_1) = \{ 0 \}$;
        \item $\ker(B_1^T) \cap \ker(A_2) \cap \ker(B_2) = \{ 0 \}$;
        \item $\ker(B_2^T) \cap \ker(A_3) = \{ 0 \}$.
    \end{enumerate}
\end{theorem}

As in the block-arrow case, a sufficient condition for invertibility when  Conditions \ref{[C1]} -- \ref{[C3]} hold is that $A_1$, $S_{t,1}$ and $S_{t,2}$ are positive definite; see \eqref{eq:tri_ldl}. We can also state the following sufficient condition when $B_1$, $B_2$ have full row rank.

\begin{theorem}[$\sT$, sufficient conditions for invertibility]
    Assume $A_1\spd$, $A_2\spsd$, $A_3\spsd$, and $B_1,B_2$ have full row rank. Assume in addition that the dimension restrictions of $\sT$ stated in Lemma \ref{lem:dimr} hold. Then $\sT$ is invertible.
\end{theorem}

\begin{proof}
     If $A_1$ is positive definite, then $A_1^{-1}$ is well-defined; if $B_1$ has full row rank and $n_2 \le n_1$, then the first Schur complement $S_{a,1} = A_2 + B_1 A_1^{-1}B_1^T$ is positive definite. Then, $S_{a,1}^{-1}$ is well-defined and the second Schur complement $S_{a,2} = A_3 + B_2 S_{a,1}^{-1}B_2^T$ is positive definite (by the assumptions that $B_2$ has full row rank and $n_3 \le n_2$). The stated result then follows from the block-LDL decomposition of $\sT$ \eqref{eq:tri_ldl} and Sylvester's Law of Inertia.
\end{proof}

For additional details on invertibility conditions for $\sT$, we refer to \cite{bgt24}.

\subsection{Eigenvalue bounds}

In this subsection we collect eigenvalue bounds for matrices in $\st$. Both block-arrow and block-tridiagonal forms are indefinite, following directly from the block-LDL decompositions derived in Section \ref{sec:blockLDL}. We state the following result on inertia, which will prove useful for establishing some later eigenvalue bounds. The results follow from the respective block-LDL decompositions of $\sA$ and $\sT$ along with Sylvester's Law of Inertia.

\begin{lemma}[Double saddle-point matrix inertia]
    If the block-arrow matrix $\sA$ is nonsingular, it has $n_1$ positive and $n_2+n_3$ negative eigenvalues. If the block-tridiagonal matrix $\sT$ is nonsingular, it has $n_1+n_3$ positive and $n_2$ negative eigenvalues.
    \label{lem:inertia}
\end{lemma}

Accordingly, we will seek upper and lower bounds on the positive and negative eigenvalues of these matrices. There has been extensive spectral analysis of the block-tridiagonal matrix $\sT$ in the literature, whereas analogous results for $\sA$ are scarcer; we therefore derive some such results in this paper. 

\paragraph{Notation} The spectral bounds derived in this section will depend on the eigenvalues of the diagonal blocks $A_1$, $A_2$, $A_3$ and the singular values of the off-diagonal blocks $B_1, B_2$. We will denote eigenvalues by $\mu$ and use the subscripts 1, 2, or 3 to denote eigenvalues of $A_1, A_2$, and $A_3$, respectively. We will denote singular values by $\sigma$ and use the subscript 1 or 2 to respectively denote singular values of $B_1$ and $B_2$. The bounds generally rely on the smallest and/or largest-magnitude singular and eigenvalues, so we will use the superscripts ``min'' and ``max'' to respectively denote the smallest and largest magnitude values. See Table \ref{tab:notation}.

\begin{table}
 \begin{center}
 \begin{tabular}{|c|c|c|c|}
 \hline
 matrix & type  & $\max$ & $\min$ \\
 \hline \hline
 $A_1$ & eigenvalues  &   $\mu_1^{\max}$ & $\mu_1^{\min}$  \\ \hline
  $A_2$ & eigenvalues & $\mu_2^{\max}$ & $\mu_2^{\min}$  \\ \hline
 $A_3$ & eigenvalues & $\mu_3^{\max}$ & $\mu_3^{\min}$  \\ \hline
 $B_1$ & singular values & $\sigma_{1}^{\max}$ &  $\sigma_{1}^{\min}$ \\ \hline
 $B_2$ & singular values & $\sigma_{2}^{\max}$ &  $\sigma_{2}^{\min}$ \\ \hline
 \end{tabular}
 \end{center}
 \caption{Summary of eigenvalues and singular values of blocks comprising matrices in $\st$}
 \label{tab:notation}
\end{table}

\subsubsection{Spectral bounds on block matrices} We begin by describing two related techniques -- energy estimates, and a technique we call the R-matrix method \cite[Chapter 1.4]{BradleyThesis2022} -- for deriving bounds on the eigenvalues of symmetric block matrices. In general, extremal bounds (upper bounds on positive eigenvalues and lower bounds on negative eigenvalues) are easier to compute than interior bounds (lower bounds on positive eigenvalues and upper bounds on negative eigenvalues).

As the setting for these methods, consider a symmetric block-$k \times k$ matrix
$$
\mathcal{M} = \begin{bmatrix}
	M_{11} & M_{21}^T &  \cdots & M_{k1}^T \\
	M_{21} & M_{22} & & M_{k2}^T\\
	\vdots & & \ddots &  \vdots \\
	M_{k1} & M_{k2} & \cdots  & M_{kk}
\end{bmatrix},
$$
and let $v = \begin{bmatrix} x_1^T & x_2^T & \cdots & x_k^T \end{bmatrix}^T$ be an appropriately partitioned block vector. Both  energy estimates and R-matrix techniques attempt to bound the quadratic form  $v^T \mathcal{M} v$, using information about the eigenvalues/singular values of the constituent blocks $M_{ij}$ of $\mathcal{M}$. Energy estimates -- as used in, for example, Rusten and Winther \cite{rw92} -- involve writing the the eigenvalue equations for $\mathcal{M}$:
\begin{align*}
	M_{11}x_1 + M_{21}^T x_2 &+ \ldots + M_{k1}^T x_k = \lambda x_1 ,\\
	&\ \ \vdots \\
	M_{k1}x_1 + M_{k2} x_2 &+ \ldots + M_{kk} x_k = \lambda x_k,
\end{align*}
and then performing various manipulations on the eigenvalue equations to obtain an inequality involving $\lambda$. The advantage of energy estimates is their generality: they can be used to obtain both extremal and interior eigenvalue bounds, though the process is generally more difficult for interior bounds. The disadvantage is that the process is open-ended and difficult to scale to larger block systems.

The R-matrix technique is only applicable for extremal bounds but is simpler to apply. We begin with the following expression for $v^T \mathcal{M} v$:

\begin{equation}
	\label{eqn:vtmv}
	v^T \mathcal{M} v = \sum_{i=1}^{k} \sum_{j=1}^{k} x_i^T M_{ij}x_j.
\end{equation}
Letting $\lambda_{ii}^{\max}$ and $\lambda_{ii}^{\min}$ respectively denote the minimal and maximal eigenvalues of $M_{ii}$ and $\sigma_{ij}^{\max}$ the maximal singular value of $M_{ij}$, we can bound each term of \eqref{eqn:vtmv} by
\begin{align*}
	\lambda_{ii}^{\min} ||x_i||^2 &\le x_i^T M_{ij}x_j \le \lambda_{ii}^{\max} ||x_i||^2 \qquad \qquad \textrm{if } i=j; \\
    -\sigma_{ij}^{\max} ||x_i||\cdot ||x_j|| &\le x_i^T M_{ij}x_j \le \sigma_{ij}^{\max} ||x_i||\cdot ||x_j|| \qquad \textrm{if } i\ne j,
\end{align*}
with the second pair of inequalities holding as a result of the Cauchy-Schwarz inequality. We can then bound \eqref{eqn:vtmv} from above by
\begin{equation*}
	v^T \mathcal{M} v \le \begin{bmatrix} ||x_1|| & ||x_2|| & \cdots & ||x_k|| \end{bmatrix} \underbrace{\begin{bmatrix}
			\lambda_{11}^{\max}& \sigma_{12}^{\max} &  \cdots & \sigma_{1k}^{\max} \\
			\sigma_{12}^{\max} & \lambda_{22}^{\max} & & \sigma_{2k}^{\max}\\
			\vdots & & \ddots &  \vdots \\
			\sigma_{1k}^{\max} & \sigma_{2k}^{\max} & \cdots  & \lambda_{kk}^{\max}
	\end{bmatrix}}_{=: R^+} \begin{bmatrix}
		||x_1|| \\
		||x_2|| \\
		\vdots \\
		||x_k||
	\end{bmatrix}.
\end{equation*}
An upper bound on $\frac{v^T \mathcal{M} v}{v^T v}$ -- and, therefore, an upper bound on the eigenvalues of $\mathcal{M}$ -- is given by the maximal eigenvalue of $R^+$, which we call the ``reduced matrix'' because it is $k \times k$ (instead of block-$k \times k$). For a block-$k \times k$ matrix, this gives us a bound that is the root of a degree-$k$ polynomial (specifically, the characteristic polynomial of $R^+$). We refer to \cite{BradleyThesis2022,bg22,sz18} for examples of this technique for classical, (block-tridiagonal) double and multiple saddle-point systems.

By similar reasoning, for a lower bound on negative eigenvalues, we can write
\begin{equation*}
	v^T \mathcal{M} v \ge \begin{bmatrix} ||x_1|| & ||x_2|| & \cdots & ||x_k|| \end{bmatrix} \underbrace{\begin{bmatrix}
			\lambda_{11}^{\min}& -\sigma_{12}^{\max} &  \cdots & -\sigma_{1k}^{\max} \\
			-\sigma_{12}^{\max} & \lambda_{22}^{\min} & & -\sigma_{2k}^{\max}\\
			\vdots & & \ddots &  \vdots \\
			-\sigma_{1k}^{\max} & -\sigma_{2k}^{\max} & \cdots  & \lambda_{kk}^{\min}
	\end{bmatrix}}_{=: R^-} \begin{bmatrix}
		||x_1|| \\
		||x_2|| \\
		\vdots \\
		||x_k||
	\end{bmatrix}.
\end{equation*}
In this case, the smallest/most negative eigenvalue of $R^{-}$ gives a lower bound on the negative eigenvalues of $\mathcal{M}$.

\subsubsection{Block-arrow eigenvalue bounds}

Because the block-arrow system is partitionable into a classical saddle-point system, we recall the following result \cite[Lemma 2.2]{sw94}.

\begin{lemma}[Eigenvalue bounds, regularized classical saddle-point systems]
    Let
    \[
    \mathcal{M} = \begin{bmatrix}
        A & B^T \\
        B & -C
    \end{bmatrix}
    \]
    denote a classical saddle-point matrix, where the eigenvalues of $A$ are between $\mu_A^{\min} >0$ and $\mu_A^{\max}$, the singular values of $B$ are between $\sigma_B^{\min} > 0$ and $\sigma_B^{\max}$, and the eigenvalues of $C$ are between $\mu_C^{\min} \ge 0$ and $\mu_C^{\max}$. Then the eigenvalues of $\mathcal{M}$ are bounded within the intervals
    {\small
    \begin{align*}
        &\left[ \frac12 \left( \mu_A^{\min} -\mu_C^{\max} - \sqrt{(\mu_A^{\min} + \mu_C^{\max})^2 + 4 (\sigma_B^{\max})^2} \right),\frac12\left( \mu_A^{\max} - \sqrt{(\mu_A^{\max})^2 + 4 (\sigma_B^{\min})^2 } \right) \right]\\
        &\qquad \qquad \qquad\cup \left[ \mu_A^{\min}, \frac12 \left( \mu_A^{\max} + \sqrt{(\mu_A^{\max})^2 + 4(\sigma_B^{\max})^2} \right) \right].
    \end{align*}}
    \label{lem:sp-eig-bounds}
\end{lemma}

We can now state the eigenvalue bounds for $\sA$. We assume below that the matrix $\begin{bmatrix} B_1 \\ B_2 \end{bmatrix}$ has full row rank, which is equivalent to requiring that $\range(B_1^T) \cap \range(B_2^T) = \{ 0\}$.

\begin{rvn}
Obtaining eigenvalue bounds for $\sA$ using Lemma \ref{lem:sp-eig-bounds} will require some knowledge of the singular values of $\mathcal{B} = \begin{bmatrix}
    B_1 \\
    B_2
\end{bmatrix}$. The following lemma will be of use.

\begin{lemma}[Singular values of $\mathcal{B}$]
    \label{lem:sv-lower-bound}
    Assume that $B_1$ and $B_2$ have full row rank, and that additionally $\mathcal{B} = \begin{bmatrix}
        B_1 \\
        B_2
    \end{bmatrix}$  has full row rank. Then the singular values of $\mathcal{B}$ are bounded from below by
    \[
    \min\{ \sigma_1^{\min}, \sigma_2^{\min} \} \cdot \sqrt{1-\cos(\theta_{\min})},
    \]
    where $\theta_{\min}$ denotes the smallest principal angle between the row space of $B_1$ and the row space of $B_2$.
\end{lemma}

\begin{proof}
    The singular values of $\mathcal{B}$ are the square roots of the $n_2+n_3$ nonzero eigenvalues of
    \[
    \mathcal{B}^T\mathcal{B} = B_1^T B_1 + B_2^T B_2.
    \]
    Let
    \[
    B_1 = U_1 \Sigma_1 V_1^T, \qquad B_2 = U_2 \Sigma_2 V_2^T
    \]
    respectively denote the reduced singular value decompositions of $B_1$ and $B_2$. Similarly to the proof of  \cite[Theorem 5]{bg23}, we can now write:
    \begin{equation}
    \mathcal{B}^T\mathcal{B} = B_1^T B_1 + B_2^T B_2 = \underbrace{\begin{bmatrix}
        V_1 & V_2
    \end{bmatrix}}_{=: V} \underbrace{\begin{bmatrix}
        \Sigma_1^2 & 0 \\
        0 & \Sigma_2^2
    \end{bmatrix}}_{=: \Sigma^2} \begin{bmatrix}
        V_1^T \\
        V_2^T
    \end{bmatrix},
    \label{eq:not-an-svd}
    \end{equation}
    where $V \in \mathbb{R}^{n_1 \times (n_2 + n_3)}$ has full row rank (because $\mathcal{B}$ has full rank) and $\Sigma \in \mathbb{R}^{(n_2 + n_3) \times (n_2 + n_3)}$ is diagonal and SPD. We note that, while Equation \eqref{eq:not-an-svd} resembles a singular value decomposition, it is not one in general because the columns of $V_1$ are not orthogonal to those of $V_2$.
    
    Because $\Sigma$ is diagonal, the nonzero eigenvalues of $\mathcal{B}^T \mathcal{B} = V\Sigma \Sigma V^T$ are the same as the eigenvalues of $W := \Sigma V^T V \Sigma$. Because $\Sigma$ and $V^T V$ are both SPD, we can write that
    \begin{align*}
    \lambda_{\min}(W) &\ge \lambda_{\min}^2(\Sigma) \cdot \lambda_{\min}(V^T V)\\
    &= \min\{ \sigma_1^{\min}, \sigma_2^{\min}\}^2 \cdot \lambda_{\min} (V^TV).
    \end{align*}
    Because $V_1$ and $V_2$ have orthonormal columns, the matrix $V^T V$ is of the form
    \[
    V^T V =
    \begin{bmatrix}
        I & V_1^T V_2 \\
        V_2^T V_1 & I
    \end{bmatrix}.
    \]
    By \cite[Theorem 7.3.3]{hj12},
    the eigenvalues of $V^T V$ are given by $\{ 1 \} \cup \{ 1-\sigma_i(V_1^T V_2) \} \cup \{ 1 + \sigma_i(V_1^T V_2) \}$. The smallest eigenvalue is therefore given by 
    \[
    1-\sigma^{\max}(V_1^T V_2) = 1 -\cos(\theta^{\min}),
    \]
    where $\theta^{\min}$ is the smallest principal angle between the subspaces spanned by $V_1$ and by $V_2$ \cite{bg73}. The stated result follows.
\end{proof}
\end{rvn}

We also define the following cubic polynomials

\begin{subequations}
	\begin{align}
    \begin{split}
		p(\lambda) =  
        &\lambda^3 - (\mu_1^{\max} - \mu_2^{\min} - \mu_3^{\min}) \lambda^2\\& - \left(\mu_1^{\max}(\mu_2^{\min} + \mu_3^{\min})- \mu_2^{\min}\mu_3^{\min}+ (\sigma_1^{\max})^2+ (\sigma_2^{\max})^2\right)\lambda\\
&- \left(
\mu_1^{\max}\mu_2^{\min}\mu_3^{\min}
+ (\sigma_1^{\max})^2 \mu_3^{\min}
+ (\sigma_2^{\max})^2 \mu_2^{\min}
\right);
\end{split}
		\label{eq:p}
	\end{align}
\begin{align}
\begin{split}
q(\lambda) =
&\lambda^3
- (\mu_1^{\min} - \mu_2^{\max} - \mu_3^{\max}) \lambda^2 \\
&- \left(
\mu_1^{\min}(\mu_2^{\max} + \mu_3^{\max})
- \mu_2^{\max}\mu_3^{\max}
+ (\sigma_1^{\max})^2
+ (\sigma_2^{\max})^2
\right)\lambda \\
&- \left(
\mu_1^{\min}\mu_2^{\max}\mu_3^{\max}
+ (\sigma_1^{\max})^2 \mu_3^{\max}
+ (\sigma_2^{\max})^2 \mu_2^{\max}
\right).
\end{split}
\label{eq:q}
	\end{align}
	\label{eq:pq}
\end{subequations}

We note that $p(\lambda)$ is the (negated) characteristic polynomial of
\begin{equation}
R_U = \begin{bmatrix}
    \mu_1^{\max} & \sigma_1^{\max} & \sigma_2^{\max} \\
    \sigma_1^{\max} & -\mu_2^{\min} & 0 \\
    \sigma_2^{\max} & 0 & -\mu_3^{\min}
    \end{bmatrix},
\label{eq:U}
\end{equation}
while $q(\lambda)$ is the negated characteristic polynomial of 
\begin{equation}
    R_L = \begin{bmatrix}
    \mu_1^{\min} & -\sigma_1^{\max} & -\sigma_2^{\max} \\
    -\sigma_1^{\max} & -\mu_2^{\max} & 0 \\
    -\sigma_2^{\max} & 0 & -\mu_3^{\max}
\end{bmatrix}.
\label{eq:L}
\end{equation}
These are both block-arrow matrices with $n_1=n_2=n_3 = 1$; provided that both matrices are invertible, the inertia result of Lemma \ref{lem:inertia} holds, implying that these matrices have one positive and two negative eigenvalues. This implies that $p(\lambda)$ and $q(\lambda)$ each have one positive real root and two negative real roots.

\begin{theorem}[Eigenvalue bounds for $\sA$]
    Assume that $B_1$ and $B_2$ have full row rank and that $\range(B_1^T) \cap \range(B_2^T) = \{ 0\}$. The eigenvalues of $\sA$ are bounded within the intervals
    \[
    \left[ q_{\min}^{-}, \rv{\frac12\left( \mu_1^{\max} - \sqrt{(\mu_1^{\max})^2 +4 \tau^2} \right) }\right] \cup \left[ \mu_1^{\min}, p^+ \right],
    \]
    \rv{where: $q_{\min}^{-}$ denotes the larger-magnitude negative root of $q(\lambda)$ \eqref{eq:q}; $p^+$ the (single)  positive root of $p(\lambda)$ \eqref{eq:p}; and $$\tau = (1-\cos(\theta_{\min})) \min\{ \sigma_1^{\min}, \sigma_2^{\min} \},$$ with $\theta_{\min}$ being the smallest principal angle between the row space of $B_1$ and the row space of $B_2$}.
    \label{thm:sa_eig}
\end{theorem}

\begin{proof}
    The extremal eigenvalue bounds are derived from the R-matrix method: for all $v = [x_1^T, x_2^T, x_3^T]$ we have that
    \begin{align*}
    v^T \sA v &\le \begin{bmatrix} ||x_1|| & ||x_2|| & ||x_3|| \end{bmatrix} \cdot  R_U \cdot \begin{bmatrix} ||x_1|| \\ ||x_2|| \\ ||x_3|| \end{bmatrix} \quad \textrm{and}\\
    v^T \sA v &\ge \begin{bmatrix} ||x_1|| & ||x_2|| & ||x_3|| \end{bmatrix} \cdot R_L \cdot \begin{bmatrix} ||x_1|| \\ ||x_2|| \\ ||x_3|| \end{bmatrix},
    \end{align*}
    with $R_U$ and $R_L$ as defined as in \eqref{eq:U} and \eqref{eq:L}, respectively; the stated results follow. 

    The interior eigenvalue bounds follow from applying Lemma \ref{lem:sp-eig-bounds} to the partitioned classical saddle-point system:
    \[
    \sA = \left[ \begin{array}{c | c c}
        A_1 & B_1^T & B_2^T \\
        \hline
        B_1 & -A_2 & 0 \\
        B_2 & 0 & -A_3
    \end{array} \right],
    \]
    and \begin{rvn}using the lower bound given by Lemma \ref{lem:sv-lower-bound} for the singular values of $\begin{bmatrix} B_1 \\ B_2 \end{bmatrix}$\end{rvn}.
\end{proof}

\begin{remark}
    We note that Theorem \ref{thm:sa_eig} requires $\range(B_1^T) \cap \range(B_2^T) = \{ 0\}$. If this does not hold, the block matrix $\begin{bmatrix}
        B_1 \\
        B_2
    \end{bmatrix}$
    has one or more singular values equal to zero, in which case the upper bound on lower negative eigenvalues obtained using Lemma \ref{lem:sp-eig-bounds} will reduce to zero.
\end{remark}

\begin{rvn}
We note that, for the unregularized matrix $\sA_0$, the polynomials $p(\lambda)$ and $q(\lambda)$ have a constant term equal to zero. Thus, the roots of these polynomials are in this case the roots of a quadratic,  which yields simpler extremal bounds in this case. These bounds are stated below, and can also be derived as a corollary of the result of Rusten and Winther \cite[Lemma 2.1]{rw92}.
\begin{corollary}[Eigenvalue bounds for $\sA_0$]
    The eigenvalues of $\sA_0$ are bounded within the intervals
    {\small
    \begin{align*}
        &\left[ \frac12 \left( \mu_1^{\min} - \sqrt{(\mu_1^{\min})^2 + 4((\sigma_1^{\max})^2 + (\sigma_2^{\max})^2)} \right),\frac12\left( \mu_1^{\max} - \sqrt{(\mu_1^{\max})^2 + 4 \tau^2 } \right) \right]\\
        &\qquad \qquad \qquad\cup \left[ \mu_1^{\min}, \frac12 \left( \mu_1^{\max} + \sqrt{(\mu_1^{\max})^2 + 4((\sigma_1^{\max})^2 + (\sigma_2^{\max})^2)} \right) \right],
    \end{align*}}
    where $\tau$ is defined as in Theorem \ref{thm:sa_eig}.
    \label{cor:sa0_eig}
\end{corollary}

\begin{example} 
We illustrate the quality of the bounds of Theorem \ref{thm:sa_eig} by testing them on a regularized version of the Stokes-Darcy equations given in Eq.~\eqref{eq:SDsymm}. We note  that the theorem requires full rank off-diagonal blocks, whereas the matrix $B_1=A_\Gamma$ is an interface matrix, and as such is rank deficient. Since our primary purpose here is to validate our theoretical results, we regularize the matrix $A_{\Gamma}$ by adding a small scalar to all its singular values. 

We compute the eigenvalues for a model problem with $n=16$; see Figure \ref{fig:eig}. 
Three of our four bounds on the eigenvalues of the block-arrow matrix in Theorem \ref{thm:sa_eig} are tight. The  exception is the upper negative bound, which is uniformly away from zero but is very close to it. In general, bounds for regularized systems with ill-conditioned or rank-deficient off-diagonal blocks often tend to be pessimistic.

\begin{figure}
    \includegraphics[scale=0.49]{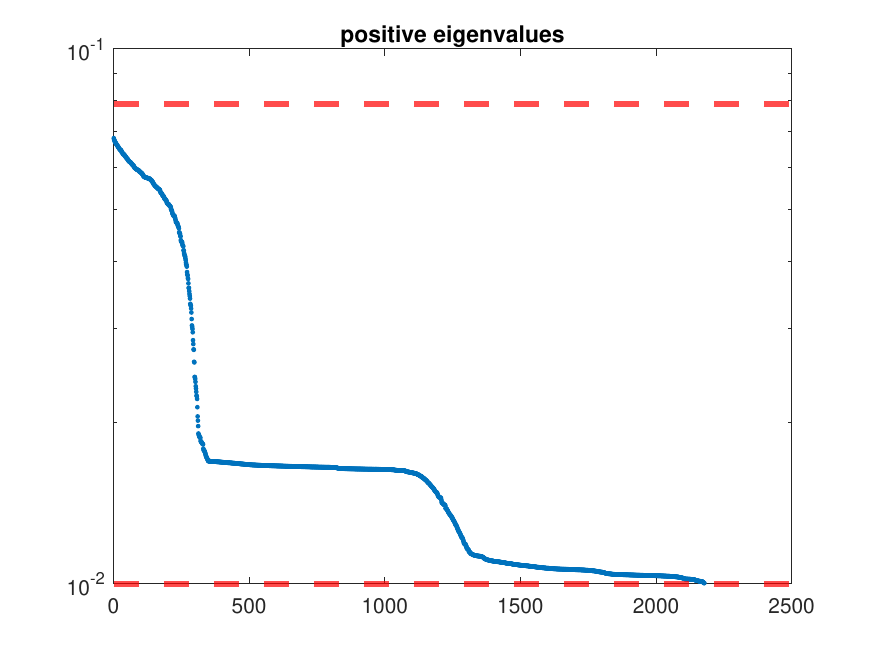}
\includegraphics[scale=0.49]{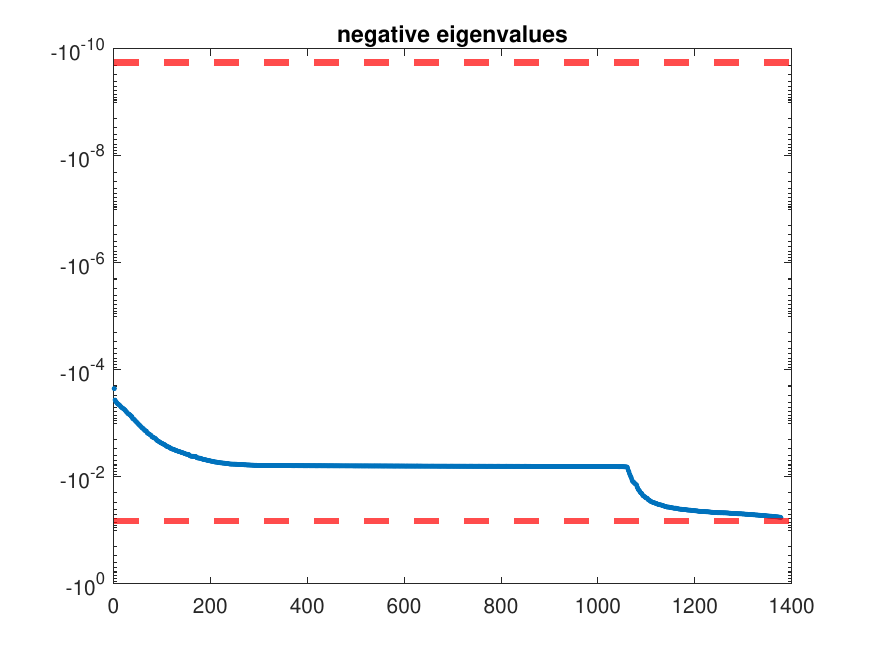}
\caption{{\color{black}An illustration of Example \ref{ex:SD}: The eigenvalues of a Stokes-Darcy matrix of dimension $3,556 \times 3,556$, shown in blue, and the corresponding bounds of Theorem \ref{thm:sa_eig}, shown in  broken red lines. }}
\label{fig:eig}
\end{figure}

\label{ex:SD}
\end{example}

\end{rvn}

\subsubsection{Block-tridiagonal eigenvalue bounds}

Eigenvalue bounds for $\sT$ are provided in \cite{bg22}; we re-state the derived results here. For a given $\sT$, we define the following cubic polynomials:

\begin{subequations}
	\begin{align}
		r(\lambda) =  \lambda^3 + (\mu_{2}^{\max} - \mu_{1}^{\min})\lambda^2 - \left(\mu_{1}^{\min} \mu_{2}^{\max} + (\sigma_{1}^{\max})^2 + (\sigma_{2}^{\min})^2 \right)\lambda + \mu_{1}^{\min}(\sigma_{2}^{\min})^2;
		\label{eq:r}
	\end{align}
	\begin{align}
		\begin{split}
			s(\lambda)  =  & \ \lambda^3 + (\mu^{\min}_{2} - \mu^{\max}_{1} - \mu^{\max}_{3})\lambda^2 \\ & +  \left( \mu^{\max}_{1}\mu^{\max}_{3} - \mu^{\max}_{2}\mu^{\min}_{2} - \mu^{\min}_{2}\mu^{\max}_{3} - (\sigma^{\max}_{1})^2 - (\sigma^{\max}_{2})^2   \right) \lambda  \\
			& + \left( \mu^{\max}_{1}\mu^{\min}_{2}\mu^{\max}_{2} + \mu_{1}^{\max}(\sigma^{\max}_{2})^2+ (\sigma^{\max}_{1})^2\mu^{\max}_{3} \right);
		\end{split}
		\label{eq:s}
	\end{align}
	
	\begin{align}
		\begin{split}
			t(\lambda) = & \ \lambda^3 + (\mu^{\max}_{2} -\mu^{\min}_{1} -  \mu^{\min}_{3})\lambda^2 \\ & + \left( \mu^{\min}_{1}\mu^{\min}_{3} - \mu^{\min}_{1}\mu^{\max}_{2} - \mu^{\max}_{2}\mu^{\min}_{3} - (\sigma^{\max}_{1})^2 - (\sigma^{\max}_{2})^2    \right) \lambda \\
			& + (\mu^{\min}_{1}\mu^{\max}_{2}\mu^{\min}_{3} + \mu_{1}^{\min}(\sigma^{\max}_{2})^2  + (\sigma^{\max}_{1})^2 \mu_{3}^{\min}).
		\end{split}
		\label{eq:t}
	\end{align}
	\label{eq:rst}
	\end{subequations}

It is shown in \cite[Corollary 2.1]{bg22} that all three of these polynomials have two positive real roots and one negative real root. We can then state the following result \cite[Theorem 2.2]{bg22}.

\begin{theorem}[Eigenvalue bounds for $\sT$]
    The eigenvalues of $\sT$ are bounded within the intervals
    \[
	\left[s^-, \frac{\mu^{\max}_{1} - \sqrt{(\mu^{\max}_{1})^2 + 4(\sigma^{\min}_{1})^2}}{2} \right] \ \bigcup \ \left[ r^+_{\min}, t^+_{\max} \right],
	\]
    where $s^-$ denotes the (single) negative root of $s(\lambda)$ \eqref{eq:s}, $r^+_{\min}$ the smaller positive root of $r(\lambda)$ \eqref{eq:r}, and $t^+_{\max}$ the larger positive root of $t(\lambda)$ \eqref{eq:t}.
\end{theorem}

We refer to \cite{bg22} for the proof of this theorem, which uses the R-matrix method to derive the two external bounds and energy estimates for the two interior bounds.


\section{Preconditioning}
\label{sec:iter-prec}
The block factorizations and spectral properties established in Section  \ref{sec:properties} provide a natural basis for the development of numerical solution methods. In this section we consider Schur complement preconditioning as an illustrative setting in which these structural features can be made explicit. Schur complements emerge directly from the block-LDL decompositions of $\st$ matrices, and their impact on spectral clustering can be understood using some of the eigenvalue bounds presented in the previous section. Our discussion therefore focuses on algebraic and spectral aspects of these preconditioners, without delving into implementation‑specific or application‑dependent details. We thus view  this section as one that provides some principles for developing block-preconditioning techniques, and not a comprehensive list of practical preconditioners, which would require 
diving deep into specific properties of the applications and the underlying discrete differential operators (when applicable). For completeness, we also include a brief overview of a few alternative preconditioning paradigms in Section \ref{sec:prec_other}.

Throughout this section, we will assume in addition to the size restrictions in Lemma \ref{lem:dimr} that the off-diagonal blocks $B_1$ and $B_2$ have full row rank. We use the notation $\mathcal{P}$ to denote a block preconditioner, with $D$ or $T$ as a subscript to denote that the preconditioner is block-diagonal or block-triangular, respectively, and $a$ or $t$ as a superscript to specify whether the preconditioner is for $\sA$ or $\sT$, respectively.

For performing eigenvalue analysis, we recall that if we use the same preconditioner, the spectra of left-preconditioned and right-preconditioned symmetric matrices are identical. In general, we will assume left preconditioning for block-diagonal preconditioners and right preconditioning for block-triangular.

\subsection{Preconditioners for $\sA$}
We begin by describing Schur complement-based preconditioners for the block-arrow matrix $\sA$. We will recall some results of Beik and Benzi \cite{bb18} and provide some new results based on the block-LDL decomposition of $\sA$ \eqref{eq:arrow_ldl}.

Because $\sA$ is, in effect, a classical saddle-point matrix, similar considerations for preconditioning hold as in the classical case -- specifically, if the matrix is unregularized (corresponding to $C=0$ in the classical saddle-point case described in Definition \ref{defn:sp}, and corresponding to $A_2 = A_3 = 0$ in the case of $\sA$), then block-diagonal Schur-complement-based preconditioners have some useful spectral properties \cite{Kuznetsov1995, mgw00}. However, with regularization ($C\ne0$ for the classical case, $A_2 \ne 0$ or $A_3 \ne 0$ in the double saddle-point case), it is more difficult to derive useful general-purpose spectral bounds with block-diagonal preconditioning, and we must, in general, consider block-triangular extensions to achieve spectral properties analogous to block-diagonal preconditioning in the unregularized case (see \cite{i01} for this discussion in the classical case).

\subsubsection{Block-diagonal preconditioners for block-arrow matrices}
Here we present some block-diagonal preconditioners for the unregularized block-arrow matrix
\begin{equation}
    \sA_0 = \begin{bmatrix}
        A_1 & B_1^T & B_2^T \\
        B_1 & 0 & 0 \\
        B_2 & 0 & 0
    \end{bmatrix}. 
\end{equation}

We consider the following preconditioner presented by Beik and Benzi \cite{bb18}:
\begin{equation}
    \mathcal{P}_{D,1}^a = \begin{bmatrix}
        A_1 & 0 & 0 \\
        0 & B_1 A_1^{-1}B_1^T & 0 \\
        0 & 0 & B_2 A_1^{-1}B_2^T
    \end{bmatrix}.
    \label{eq:beik-benzi}
\end{equation}
We note that this preconditioner does not arise from the block-$3 \times 3$ LDL factorization of $\sA$ \eqref{eq:arrow_ldl}, but is rather obtained by considering the block-diagonal Schur complement preconditioner for classical saddle-point systems \cite{Kuznetsov1995, mgw00} applied to $\sA_0$ with classical saddle-point partitioning, and then setting the (2,3)- and (3,2)-blocks of the preconditioner to 0. Beik and Benzi prove the following result \cite[Theorem 4.1]{bb18}.

\begin{theorem}[Eigenvalue bounds, $(\mathcal{P}_{D,1}^{a})^{-1} \sA_0$]
    \label{thm:beik_benzi_prec_bounds}
    Suppose $A_1 \succ 0$ and $B_1$ and $B_2$ have full row rank. Then the eigenvalues of $(\mathcal{P}_{D,1}^{a})^{-1} \sA_0$ lie in the union of intervals

    \[
    \left( -1, \frac{1-\sqrt{1+4\gamma^*}}{2} \right) \cup \{ 1 \} \cup \left( \frac{1+\sqrt{1+4\gamma^*}}{2}, 2\right),
    \]
    where
    \rv{
    \[
    \gamma^* = \min\frac{x^T \left( B_1^T (B_1A_1^{-1}B_1^T)^{-1} B_1 + B_2^T (B_2A_1^{-1}B_2^T)^{-1} B_2 \right)x}{x^TA_1x} > 0.
    \]}
\end{theorem}

We now consider the analogous block-diagonal preconditioner we obtain from the block-$3 \times 3$ block-LDL decomposition \eqref{eq:arrow_ldl}. We begin with the positive definite block-diagonal preconditioner obtained by taking the block diagonal  term of the block-LDL decomposition, with all negative semidefinite blocks negated to create a positive definite preconditioner
$$
\mathcal{P}_{D,2}^{a} = \begin{bmatrix}
    A_1 & 0 & 0 \\
    0 & S_{a,1} & 0 \\
    0 & 0 & S_{a,2}
\end{bmatrix},
$$
with $S_{a,1} = B_1A_1^{-1}B_1^T$ and \rv{$S_{a,2} = B_2A_1^{-1}B_2^T - B_2A_1^{-1}B_1^T S_{a,1}^{-1}B_1A_1^{-1}B_2^T$} (we note that $S_{a,1}$ and $S_{a,2}$ here do not have an additive $A_2$ and $A_3$ term, respectively, because we are considering preconditioning for the system $\sA_0$). We recall that $S_{a,2}$ is guaranteed to be positive semidefinite, and will be positive definite if $\sA_0$ is invertible (see Lemma \ref{prop:spd}); for purposes of this discussion, we are assuming that the block-arrow system is invertible and that all Schur complement terms are thus positive definite.
The following result then holds; note that because we are assuming the unregularized block-arrow system $\sA_0$ we can assume without loss of generality that $n_2 \ge n_3$. 
\begin{theorem}[Eigenvalues of $(\mathcal{P}_{D,2}^a)^{-1} \sA_0$]
Assume without loss of generality that $n_2 \ge n_3$. The preconditioned matrix $(\mathcal{P}_{D,2}^a)^{-1} \sA_0$ has eigenvalue $\lambda = 1$ with multiplicity \rv{at least} $n_1-(n_2+n_3)$, and eigenvalues $\lambda = \frac{1 \pm \sqrt{5}}{2}$, each with multiplicity of at least $n_2-n_3$. \rv{In addition, the positive eigenvalues of the preconditioned matrix are bounded from below by $1$.}
\label{thm:pd2a0}
\end{theorem}

\begin{proof}
    We begin by writing the eigenvalue equations:
    \begin{subequations}\label{eq:eig_blocktri_noreg_prec}
        \begin{align}
            A_1 x + B_1^T y + B_2^T z &= \lambda A_1 x; \label{eq:eig_a1x}\\ 
            B_1 x &= \lambda S_{a,1} y; \label{eq:eig_sa1y}\\ 
            B_2 x &= \lambda S_{a,2} z. \label{eq:eig_sa2z}
        \end{align}
    \end{subequations}
 The stated result for the eigenvalue $\lambda =1$ follows from setting $x \in \ker\left(\begin{bmatrix} B_1 \\ B_2 \end{bmatrix} \right)$, $y=0$, $z=0$.
 
\rv{
Next, we use \eqref{eq:eig_sa1y} to write $y=\frac{1}{\lambda} S_{a,1}^{-1} B_1 x$  and \eqref{eq:eig_sa2z} to write $z=\frac{1}{\lambda} S_{a,2}^{-1} B_2 x$ and substitute these expressions for $y$ and $z$ into \eqref{eq:eig_a1x}:
\begin{equation}
    (1-\lambda) A_1 x + \frac{1}{\lambda} B_1^T S_{a,1}^{-1} B_1 x + 
    \frac{1}{\lambda} B_2^T S_{a,2}^{-1} B_2 x = 0.
    \label{eq:eigx}
\end{equation}
}
Consider $x \in \ker(B_2) \setminus \ker(B_1)$. Using this and multiplying \eqref{eq:eigx} by $A_1^{-1}$ on the left gives us
    \[
    (1-\lambda)x + \frac{1}{\lambda} A_1^{-1}B_1^T S_{a,1}^{-1}B_1 x  = 0.
    \]
    Observe that, when $S_{a,1} = B_1A_1^{-1}B_1^T$, the matrix $A_1^{-1}B_1^T S_{a,1}^{-1}B_1$ is a projector onto $\range(A_1^{-1}B_1^T)$.
    For $x$ in this range, the equation for $\lambda$ simplifies to
    \[
    \lambda - 1 - \frac{1}{\lambda} = 0,
    \]
    yielding the values $\lambda = \frac{1\pm \sqrt{5}}{2}$.

    The number of these eigenvalues is given by \rv{at least} the dimension of the subspace $\range(A_1^{-1} B_1^T) \cap \ker(B_2)$ (because we required $x \in \ker(B_2)$ in this case). We observe that a vector $x$ in $\range(A_1^{-1} B_1^T)$ can be written as $x = A_1^{-1} B_1^Ty$, for some $y \in \mathbb{R}^{n_2}$. Such a vector lies in $\ker(B_2)$ if and only if
    \[
    B_2 A_1^{-1}B_1^T y = 0.
    \]
    Thus, the algebraic multiplicity of the eigenvalues is given by the nullity of the matrix \\ $B_2 A_1^{-1}B_1^T \in \mathbb{R}^{n_3 \times n_2}$, which is at least $n_2 - n_3$.

\begin{rvn}
Next, we show that 
 $x \notin \ker(B_1) \setminus \ker(B_2)$, that 
 is, we cannot have a situation whereby $x$ is in the kernel of $B_1$ and not in the kernel of $B_2$. Indeed, multiplying \eqref{eq:eigx} by $B_1 A_1^{-1}$, 
 we obtain
 \[
    \left(1-\lambda+ \frac{1}{\lambda}\right) B_1 x + \frac{1}{\lambda} B_1 A_1^{-1}B_2^T S_{a,2}^{-1}B_2 x  = 0.
    \] If $x \in \ker(B_1)$, the first term in the above equation vanishes, and we are left with $\frac{1}{\lambda} B_1 A_1^{-1}B_2^T S_{a,2}^{-1}B_2 x  = 0$,
    which cannot generate a valid eigenpair $(\lambda,x)$ .

The remaining eigenvectors have components $x$ that are thus neither in $\ker(B_1)$ nor in $\ker(B_2)$. Solving the eigenvalue problem results in relations that are difficult to decouple. We can still make a useful observation on the positive eigenvalues: It is immediate to observe that the preconditioned matrix $(\mathcal{P}^a_{D,2})^{-1}\mathcal{A}_0$ is a block-arrow matrix whose leading block is the $n_1 \times n_1$ identity matrix, and therefore by Theorem \ref{thm:sa_eig}, its positive eigenvalues are bounded below by  1. 
\end{rvn}
\end{proof}

\rv{\begin{remark} Theorem \ref{thm:pd2a0}  shows that, contrary to the situation with the block-tridiagonal form, for the block-arrow form the block-diagonal preconditioner does not guarantee that all eigenvalues are bounded uniformly away from zero. We cannot  determine the complete set of the eigenvalues, and the theorem does not preclude the possibility of some negative eigenvalues being close to zero, which may slow down the convergence of MINRES and other iterative solvers. Nonetheless, under additional specific assumptions on the blocks of the linear system, the remaining eigenvalues may be further bounded. \end{remark}}

\begin{rvn}
We now consider the matrix $\sA$  with regularization. Using the analogous preconditioner to $\sP_{D,2}^a$ yields no results of interest. Instead, we  offer the following preconditioner, which generalizes that of \cite{bb18}, Eq. \eqref{eq:beik-benzi}, and amounts to dropping the subtractive last term of $S_{a,2}$:
\begin{equation}
    \mathcal{P}_{D,3}^{a} = \begin{bmatrix}
        A_1 & 0 & 0\\
        0 & A_2 + B_1A_1^{-1}B_1^T & 0 \\
        0 & 0 & A_3 + B_2 A_1^{-1}B_2^T
    \end{bmatrix}.
\end{equation}

\begin{theorem}[Eigenvalue bounds, $(\sP_{D,3}^{a})^{-1}\sA$]
\label{thm:prec_eig_new}
Suppose $A_1 \succ 0$ and $B_1$, $B_2$ have full row rank. Suppose further that
\[
\beta = (1- \cos(\phi_{\min})) \min\{ \beta_1^{\min}, \beta_2^{\min} \},
\]
where $\beta_1^{\min}$ and $\beta_2^{\min}$ respectively denote the minimal singular values of $S_{a,1}^{-1/2}B_1A_1^{-1/2}$ and $\bar{S}_{a,2}^{-1/2} B_2 A_1^{-1/2}$, and $\phi_{\min}$ denotes the smallest principal angle between the row spaces of $S_{a,1}^{-1/2}B_1A_1^{-1/2}$ and $\bar{S}_{a,2}^{-1/2} B_2 A_1^{-1/2}$.
 Then,  the eigenvalues of $(\sP_{D,3}^{a})^{-1}\sA$ lie in the union of intervals:
\[
\left[ -1, \frac12\left( 1 - \sqrt{1 +4 \beta^2} \right) \right] \cup \left[ 1, 2 \right).
\]  
Furthermore, there is an eigenvalue equal to $1$, whose algebraic multiplicity is $n_1-(n_2+n_3)$.
\end{theorem}

\begin{proof}
    We define $\bar{S}_{a,2} = A_3 + B_2A_1^{-1}B_2^T$ and write
    \[
    \mathcal{P}_{D,3}^{a} = \begin{bmatrix}
        A_1 & 0 & 0\\
        0 & S_{a,1} & 0 \\
        0 & 0 & \bar{S}_{a,2}
    \end{bmatrix}.
    \]  
    The extremal bounds are obtained via energy estimates. Let us write out the eigenvalue problem explicitly:
 \[ \begin{bmatrix}
     A_1 & B_1^T & B_2^T \\ 
     B_1 & -A_2 \\
     B_2 & & -A_3
 \end{bmatrix}
     \begin{bmatrix}
        x\\
        y \\
        z 
    \end{bmatrix}=
    \lambda
    \begin{bmatrix}
        A_1 & 0 & 0\\
        0 & A_2+B_1 A_1^{-1} B_1^T & 0 \\
        0 & 0 & A_3+B_2 A_1^{-1} B_2^T
    \end{bmatrix}
 \begin{bmatrix}
        x\\
        y \\
        z 
    \end{bmatrix} .
    \]   
The second and third block rows give us  $$ y= {\underbrace{\left[ (\lambda+1) A_2 +\lambda B_1 A_1^{-1} B_1^T \right]}_{M_1}}^{-1} B_1 x \ ; \quad
z ={\underbrace{\left[ (\lambda+1) A_3 +\lambda B_2 A_1^{-1} B_2^T \right]}_{M_2}}^{-1} B_2 x.$$ Substituting these expressions for $y$ and $z$ into the first block row, we obtain 
 \begin{equation} (1-\lambda) A_1 x + B_1^T M_1^{-1} B_1 x + B_2^T M_2^{-1} B_2 x =0.
 \label{eq:M1}
 \end{equation}

 \noindent
{\bf Upper bound on positive eigenvalues.} Suppose $\lambda > 0.$ Then, 
\begin{eqnarray*}
 M_1 & = & (\lambda+1) A_2 + \lambda B_1 A_1^{-1} B_1^T \\ 
 &  \succeq & \lambda B_1 A_1^{-1} B_1^T \succ 0, 
 \end{eqnarray*}
 and therefore $0 \prec M_1^{-1} \preceq \frac{1}{\lambda} (B_1 A_1^{-1} B_1^T)^{-1}.$ We have an analogous result for $M_2$.
Left-multiplying \eqref{eq:M1} by $x^T A_1^{-1}$, followed by setting $\tilde{x} =  A_1^{-1/2} x$, we get
$$ (1-\lambda) \tilde{x}^T \tilde{x} + \frac{1}{\lambda}  \tilde{x}^T P_1 \tilde{x} + \frac{1}{\lambda} \tilde{x}^T P_2 \tilde{x}  \ge 0,$$
where 
$$ P_i = A_1^{-1/2} B_i^T (B_i A_1^{-1} B_i^T)^{-1}  B_i A_1^{-1/2}, \qquad i=1, 2,$$
are orthogonal projectors onto the ranges of $A_1^{-1/2} B_i^T.$ 
Using the trivial operator bound \(0\preceq P_1+P_2\preceq 2I\)  yields
\[
(1-\lambda)+\frac{2}{\lambda}\ge0 \quad\Longrightarrow\quad \lambda^2-\lambda-2\le0,
\]
which gives  $\lambda \le 2$.

 \noindent
{\bf Lower bound on negative eigenvalues.} 
Left-multiplying \eqref{eq:M1} by $x^T A_1^{-1}$, followed by setting $\tilde{x} = A_1^{-1/2} x$, 
yields 
\begin{equation}
(1-\lambda)\tilde x^T \tilde x + \tilde x^T (B_1A_1^{-1/2})^T M_1^{-1}(B_1A_1^{-1/2})\tilde x
+ \tilde x^T (B_2A_1^{-1/2})^T M_2^{-1}(B_2A_1^{-1/2})\tilde x = 0.
\label{eq:messy}
\end{equation}
We will prove that if $\lambda$ is negative, then it is bounded from below by $-1$. Indeed, suppose by contradiction that this is not the case and $\lambda < -1.$
Then, $M_i, i=1, 2$ are symmetric negative definite. We have
\begin{eqnarray*}
 M_1 & = & (\lambda+1) A_2 + \lambda B_1 A_1^{-1} B_1^T \\ 
 &  \preceq & \lambda B_1 A_1^{-1} B_1^T.
 \end{eqnarray*}
 We thus get $0 \succ M_1^{-1} \succeq \frac{1}{\lambda} (B_1 A_1^{-1} B_1^T)^{-1}.$ Similarly, for $M_2$ we have $0 \succ M_2^{-1} \succeq \frac{1}{\lambda} (B_2 A_1^{-1} B_2^T)^{-1}.$
We now apply the above inequalities  to simplify \eqref{eq:messy}. We obtain a flipped inequality compared to the previous case:
$$ (1-\lambda) \tilde{x}^T \tilde{x} + \frac{1}{\lambda}  \tilde{x}^T P_1 \tilde{x} + \frac{1}{\lambda} \tilde{x}^T P_2 \tilde{x}  \le 0.$$
The inequality 
$
(1-\lambda)+\frac{2}{\lambda} \le 0 
$
implies $\lambda \in [-1,2]$, and this contradicts the assumption $\lambda<-1.$
Hence $\lambda \ge -1$, as required.

For the interior bounds, we form the split preconditioned system:
    \[
\mathcal{P}_{D,3}^{-1/2}\sA\mathcal{P}_{D,3}^{-1/2} = \begin{bmatrix}
        I & A_1^{-1/2}B_1^TS_{a,1}^{-1/2} & A_1^{-1/2}B_2^T-\bar{S}_{a,2}^{-1/2} A_3 \bar{S}_{a,2}^{-1/2} \\
        S_{a,1}^{-1/2}B_1A_1^{-1/2} & -S_{a,1}^{-1/2}A_2 S_{a,1}^{-1/2} & 0 \\
        \bar{S}_{a,2}^{-1/2} B_2 A_1^{-1/2} & 0 & -\bar{S}_{a,2}^{-1/2} A_3 \bar{S}_{a,2}^{-1/2}
    \end{bmatrix}.
    \]
    This is double saddle-point matrix in block-arrow form. The stated upper negative and lower positive eigenvalue bounds then follow immediately from Theorem \ref{thm:sa_eig}.
\end{proof}

\begin{remark}
    While the positive eigenvalues of $(\sP_{D,3}^{a})^{-1}\sA$ are bounded in the interval $[1,2]$ and the negative eigenvalues are bounded from below by $-1$, Theorem \ref{thm:prec_eig_new} does not allow us to state a constant upper bound on the negative eigenvalues of $(\sP_{D,3}^{a})^{-1}\sA$, and these may indeed be arbitrarily close to zero. This is similar to the eigenvalue bounds in \cite[Theorem 4.1]{bb18} for this system when $A_2$ and $A_3$ are zero; see Theorem \ref{thm:beik_benzi_prec_bounds}. 
\end{remark}

\subsubsection{Approximations of Schur complements}
While Schur complement-based preconditioners can improve the spectral properties of the matrices being solved, the $A_i$ and Schur complement terms appearing in these terms are typically too expensive to solve (or, in the case of the Schur complements, to form) exactly. Accordingly, a useful question to consider in these cases is how the preconditioned eigenvalue bounds are affected when one or more of these terms is approximated. These kinds of analyses have been performed for the block-tridiagonal system in, for example, \cite{bg22}, which analyzed the effect of spectrally equivalent approximations. Here, we adapt the result of Theorem \ref{thm:beik_benzi_prec_bounds} in a similar fashion as \cite{bg22}: we will assume spectrally equivalent approximations, derive bounds on the eigenvalues/singular values of the blocks of the preconditioned operator, and use the established eigenvalue bounds for unpreconditioned systems.

We consider the unregularized block-arrow matrix $\sA_0$ and the block-diagonal preconditioner
\[
\tilde{\sP}_{D,1}^a = \begin{bmatrix}
    \tilde{A}_1 & 0 & 0 \\
    0 & \tilde{S}_1 & 0 \\
    0 & 0 & \tilde{S}_2
\end{bmatrix},
\]
where $\tilde{A}_1$, $\tilde{S}_1$, and $\tilde{S}_2$ are spectrally equivalent approximations of $A_1$, $S_1 := B_1A_1^{-1}B_1^T$, and $S_2 := B_2A_1^{-1}B_1^T$, respectively.

More formally, we assume the following.

\begin{assumption}
	\label{thm:spec_equiv_unreg}
	Let $\Lambda(\cdot)$ denote the spectrum of a matrix. The diagonal blocks of the approximate preconditioner $\tilde{\sP}$, given by $\tilde{A}$, $\tilde{S}_1$, and $\tilde{S}_2$, satisfy:
	\begin{enumerate}[label=\roman*.]
		\item $\Lambda(\tilde{A}^{-1}A) \in [\alpha_0, \beta_0]$;
		\item $\Lambda(\tilde{S}_1^{-1}S_1) \in [\alpha_1, \beta_1]$;
		\item $\Lambda(\tilde{S}_2^{-1}S_2) \in [\alpha_2, \beta_2]$,
	\end{enumerate}
	where $0 < \alpha_i \le 1 \le \beta_i$.
\end{assumption}

To simplify our analyses, we define:
\begin{subequations}
	\begin{align}
	\label{eqn:q0}
	\tilde{A}^{-1}A =: Q_0; \\
	\label{eqn:q1}
	\tilde{S}_1^{-1}S_1 =: Q_1; \\
	\label{eqn:q2}
	\tilde{S}_2^{-1}S_2 =: Q_2.
	\end{align}
\end{subequations}
The split preconditioned operator is
\begin{align}
    (\tilde{\sP}_{D,1}^a)^{-1/2} \sA_0 (\tilde{\sP}_{D,1}^a)^{-1/2} &= \begin{bmatrix}
        \tilde{A}_1^{-1/2}A_1\tilde{A}_1^{-1/2} & \tilde{A}_1^{-1/2}B_1^T \tilde{S}_1^{/1-2} & \tilde{A}_1^{-1/2}B_2^T \tilde{S}_S^{-1/2} \\
        \tilde{S}_1^{-1/2} B_1 \tilde{A}_1^{-1/2} & 0 & 0 \\
        \tilde{S}_2^{-1/2} B_2 \tilde{A}_1^{-1/2} & 0 & 0
    \end{bmatrix} \\
    &=: \begin{bmatrix}
        \tilde{Q}_0 & \tilde{B}_1^T & \tilde{B}_2^T \\
        \tilde{B}_1 & 0 & 0 \\
        \tilde{B}_2 & 0 & 0
    \end{bmatrix}.
\end{align}
This matrix is itself an unregularized block-arrow matrix; thus, we can derive bounds on the spectrum of this matrix provided we have bounds for the eigenvalues/singular values of the constituent blocks. These are given in the following lemma.

\begin{lemma}[Eigenvalue/singular value bounds, blocks of $(\tilde{\sP}_{D,1}^a)^{-1/2} \sA_0 (\tilde{\sP}_{D,1}^a)^{-1/2}$]
    The eigenvalues of $\tilde{Q}_0$ satisfy $\Lambda(\tilde{Q}_0) \in [\alpha_0, \beta_0]$. The singular values of $\tilde{B}_1$ satisfy $\sigma(\tilde{B}_1) \in [\sqrt{\alpha_0 \alpha_1}, \sqrt{\beta_0 \beta_1}]$, and those of 
    $\tilde{B}_2$ satisfy $\sigma(\tilde{B}_2) \in [\sqrt{\alpha_0 \alpha_2}, \sqrt{\beta_0 \beta_2}]$.
\end{lemma}

\begin{proof}
    The bounds on the eigenvalues of $\tilde{Q}_0$ follow immediately from the observation that $\tilde{Q}_0$ is similar to $Q_0$.

    Now let us consider the singular values of $\tilde{B}_1$. These are given by the square roots of the eigenvalues of $\tilde{B}_1 \tilde{B}_1^T = \tilde{S}_1^{-1/2} B_1 \tilde{A}_1^{-1} B_1^T \tilde{S}_1^{-1/2}$. We note that this matrix is similar to
    \[
    \tilde{S}_1^{-1} B_1 \tilde{A}_1^{1} B_1^T = Q_1 S_1^{-1} B_1 \tilde{A}_1^{1} B_1^T.
    \]
    Because $Q_1$ and $S_1^{-1} B_1 \tilde{A}_1^{1} B_1^T$ are both SPD, the eigenvalues of the product can be bounded from below by $\alpha_1$ times the minimum eigenvalue of $S_1^{-1} B_1 \tilde{A}_1^{1} B_1^T$ and from above by $\beta_1$ times the maximum eigenvalue of $S_1^{-1} B_1 \tilde{A}_1^{1} B_1^T$.

    We note that the eigenvalues of $S_1^{-1} B_1 \tilde{A}_1^{-1} B_1^T$ are the same as the nonzero eigenvalues of $\tilde{A}_1^{-1} B_1^T S_1^{-1}B_1 = Q_0 A_1^{-1} B_1^T S_1^{-1} B_1$. The matrix $A_1^{-1} B_1^T S_1^{-1}B_1$ is a rank-$n_2$ projector, so its nonzero eigenvalues are equal to 1. Thus, the nonzero eigenvalues of $Q_0 A_1^{-1} B_1^T S_1^{-1} B_1$ (and, therefore the eigenvalues of $S_1^{-1} B_1 \tilde{A}_1^{-1} B_1^T$) are contained in $[\alpha_0, \beta_0]$.

    This gives us the stated range $[\sqrt{\alpha_0 \alpha_1}, \sqrt{\beta_0 \beta_1}]$ for the singular values of $\tilde{B}_1$. Identical reasoning yields the stated singular value bounds for $\tilde{B}_2$.
\end{proof}

Combining the results of this lemma with the result of Corollary \ref{cor:sa0_eig} yields the following.

\begin{theorem}[Eigenvalue bounds, $(\tilde{\sP}_{D,1}^a)^{-1} \sA_0$] 
    The eigenvalues of $(\tilde{\sP}_{D,1}^a)^{-1} \sA_0$ are contained in the union of intervals:
    {\small
    \begin{align*}
        &\left[ \frac12 \left( \alpha_0 - \sqrt{\alpha_0^2 + 4(\beta_0\beta_1 + \beta_0 \beta_2)} \right),\frac12\left( \beta_0 - \sqrt{\beta_0^2 + 4 \tilde{\tau}^2 } \right) \right]\\
        &\qquad \qquad \qquad\cup \left[ \alpha_0, \frac12 \left( \alpha_1 + \sqrt{\alpha_1^2 + 4(\beta_0\beta_1 + \beta_0 \beta_2)} \right) \right],
    \end{align*}}
    where 
    \[
    \tilde{\tau} = (1 -\cos(\tilde{\theta}_{\min})) \min \{ \sqrt{\alpha_0 \alpha_1}, \sqrt{\alpha_0 \alpha_2} \},
    \]
    with $\tilde{\theta}_{\min}$ being the smallest principal angle between the row space of $\tilde{B}_1 = \tilde{S}_1^{-1/2} B_1 \tilde{A}_1^{-1/2}$ and $\tilde{B}_2 = \tilde{S}_2^{-1/2} B_2 \tilde{A}_1^{-1/2}$.
\end{theorem}
\end{rvn}

\subsubsection{Block-triangular preconditioners for $\sA$}
We can  consider a block-triangular preconditioner by multiplying two factors of the block-LDL decomposition \eqref{eq:arrow_ldl} -- either $\mathcal{D}_a \mathcal{L}_a^T$ (for an upper triangular preconditioner) or $\mathcal{L}_a \mathcal{D}_a$ (for a lower-triangular preconditioner) -- to obtain a preconditioner analogous to the block-triangular Schur complement preconditioner for the classical saddle-point system  \cite{i01}. Without loss of generality, we consider the upper triangular preconditioner
\begin{equation}
    \mathcal{P}_{T,1}^a := \mathcal{D}_a\mathcal{L}_a^T = \rv{\begin{bmatrix}
        A_1 & B_1^T & B_2^T \\
        0 & -S_{a,1} & -B_1A_1^{-1}B_2^T \\
        0 & 0 & -S_{a,2}
    \end{bmatrix}}
\end{equation}
with right preconditioning. The following result holds.

\begin{proposition}[Eigenvalues of $ \sA (\mathcal{P}_{T,1}^a)^{-1}$]
    The preconditioned operator \\ $\sA (\mathcal{P}_{T,1}^a)^{-1}$ has minimal polynomial $(\lambda - 1)^3$.  
\end{proposition}

\begin{proof}
The proof follows from the block-LDL factorization of $\sA$. Note that 
\[
\sA  (\mathcal{P}_{T,1}^a)^{-1} = \sA (\mathcal{D}_a\mathcal{L}_a^T)^{-1} = \mathcal{L}_a,
\]
where
\[
\mathcal{L}_a = \begin{bmatrix}
    I & 0 & 0 \\
    B_1A_1^{-1} & I & 0 \\
    B_2A_1^{-1} & B_2A_1^{-1}B_1^T S_{a,1}^{-1} & I
\end{bmatrix}.
\]
It is immediate to verify  $(\mathcal{L}_a - I)^3 = 0$.
\end{proof}

The preconditioner $\mathcal{P}_{T,1}^a$ is clearly impractical, even as a theoretical baseline: the fill-in term $-B_1A_1^{-1}B_2^T$ in the (2,3)-block in addition to the computationally expensive $S_{a,1}$ and $S_{a,2}$ terms makes this preconditioner very expensive. Rather than proposing $\mathcal{P}_{T,1}^a$ for algorithmic use, we include this result primarily as a point of reference: it characterizes the idealized block-triangular preconditioner associated with the exact block‑LDL factorization, and provides a baseline against which more practical (inexact or approximate) triangular preconditioners can be compared.

Beik and Benzi \cite{bb18} propose the following block-triangular preconditioner (among other preconditioners they derive) for the case in which $A_2 = 0$, $A_3 \succeq 0$:
\begin{align*}
    \mathcal{P}_{T,2}^a &= \begin{bmatrix}
        A_1 & B_1^T & B_2^T \\
        0 & -B_1A_1^{-1}B_1^T & -B_1A_1^{-1}B_2^T \\
        0 & 0 & -(A_3 + B_2A_1^{-1}B_2^T)
    \end{bmatrix}.
\end{align*}
They show that 
that the eigenvalues of 
the preconditioned matrix
are contained in the interval $(0,2)$. (They in fact show that the eigenvalues are bounded away from zero, but the bound depends on  properties of the blocks of $\sA$; we refer to \cite[Theorem 4.7]{bb18} for details.)

\subsection{Preconditioners for $\sT$}
We now present preconditioners for the block-tridiagonal matrix $\sT$. Unlike in the block-arrow case, the positive definite block-diagonal preconditioner obtained from the block-LDL decomposition yields a preconditioned operator with constant spectral bounds, even when $A_2$ and $A_3$ are not zero. Extensive analysis has been performed on preconditioners for $\sT$; we refer to, for example, \cite{BergamaschiTriangular2025,bmpp24,bg22,cjl21,cmx09,hm19,pp24,sz18}.

\subsubsection{Block-diagonal preconditioners for $\sT$}
The most common algebraic preconditioner used for the block-tridiagonal system $\sT$ is the block-diagonal Schur complement preconditioner
\begin{equation}
    \label{eq:pre_blockdi_kt}
    \mathcal{P}_D^t = \begin{bmatrix}
        A & 0 & 0 \\
        0 & S_{t,1} & 0 \\
        0 & 0 & S_{t,2}
    \end{bmatrix}
\end{equation}
where $S_{t,1} = A_2 + B_1A_1^{-1}B_1^T$ and \rv{$S_{t,2} = A_3 + B_2 S_{t,1}^{-1} B_2^T$}. This preconditioner comes from the block-diagonal term of the block-LDL factorization of $\sT$ given in \eqref{eq:tri_ldl}, and can be viewed as a natural extension of the preconditioner \cite{Kuznetsov1995, mgw00} for classical saddle-point systems. This preconditioner has been analyzed, for example, in \cite{cmx09,sz18,hm19,bg22,bmpp25}.

Several theoretical results have been established for the spectral properties of \((\mathcal{P}_D^t)^{-1}\sT\). The bounds vary depending on whether $A_2$ and/or $A_3$ are equal to zero.

We begin with the case in which both $A_2$ and $A_3$ are zero. Recall that in this case we require that $n_2 \ge n_3$ and for $B_2$ to have full row rank in order for $\sT$ to be nonsingular.

\begin{theorem}[Eigenvalues of $(\mathcal{P}_D^t)^{-1}\sT$ when $A_2 =0, A_3=0$]
    When $A_2=0$ and $A_3=0$, the preconditioned matrix $(\mathcal{P}_D^t)^{-1}\sT$ has six distinct eigenvalues:
    \begin{itemize}
        \item $\lambda = 1$ with multiplicity $n_1-n_2$;
        \item $\lambda = \frac{1 \pm \sqrt{5}}{2}$, each with multiplicity $n_2-n_3$;
        \item \( \lambda = 2\cos\!\big(\tfrac{\pi}{7}\big)\;(\approx 1.802),\;
2\cos\!\big(\tfrac{3\pi}{7}\big)\;(\approx 0.445),\;
2\cos\!\big(\tfrac{5\pi}{7}\big)\;(\approx -1.247) \), each with multiplicity $n_3$. (These values are roots of the cubic polynomial $\lambda^3 - \lambda^2 -2\lambda +1$.)
    \end{itemize}
\end{theorem}

\begin{proof}
    See \cite[Theorem 2.3]{sz18}, \cite[Theorem 2.1]{hm19}, and \cite[Theorem 2.2]{cjl21}.
\end{proof}

When $A_2 = 0$ and $A_3 \succeq 0$, we require $n_2 \ge n_3$ and for $B_1$ to have full row rank in order for $S_{t,1}$ to be invertible. However, we may have a rank-deficient $B_2$ and still have nonsingular $\sT$ and $S_{t,2}$. Let $k$ denote the nullity of $B_2$. The following result is shown in \cite[Theorem 3.2]{bg22}.

\begin{theorem}[Eigenvalues of $(\mathcal{P}_D^t)^{-1}\sT$ when $A_2 = 0$, $A_3 \succeq 0$]
	\label{thm:bounds_prec_c10}
	When $A_2 = 0$ and $A_3 \succeq 0$, the eigenvalues of $(\mathcal{P}_D^t)^{-1}\sT$ are given by: $\lambda = 1$ with multiplicity $n_1-n_2+k$; $\lambda = \frac{1 \pm \sqrt{5}}{2}$, each with multiplicity $n_2-n_3+k$; and $n_3-k$ eigenvalues located in each of the three intervals:
	\begin{itemize}
	\item $I_1 = \left[ 2\cos\left( \frac{5\pi}{7} \right), \frac{1 - \sqrt{5}}{2} \right) \approx [-1.247,-0.618)$
	\item $I_2 = \left[ 2\cos\left( \frac{3\pi}{7} \right), 1 \right) \approx [0.445, 1) $
	\item $I_3 = \left( \frac{1 + \sqrt{5}}{2}, 2\cos\left( \frac{\pi}{7} \right) \right] \approx (1.618,1.802].$
	\end{itemize}
\end{theorem}

Bounds for the case in which $A_2 \succeq 0$, $A_3 = 0$ are the same as when $A_2, A_3 \succeq 0$. Bounds for this case are given in the following result.

\begin{theorem} [Eigenvalue bounds,  matrix $(\mathcal{P}_D^t)^{-1}\sT$, $A_2, A_3 \succeq 0$]
\label{thm:bounds_triSC_prec}
The eigenvalues of $(\mathcal{P}_D^t)^{-1}\sT$ are bounded within the intervals 
$$
\left[ -\frac{1+\sqrt5}{2}, \frac{1-\sqrt{5}}{2}\right] \cup \left[ 2\cos\left( \frac{3\pi}{7} \right), 2\cos\left( \frac{\pi}{7} \right)\right],
$$
which are approximately
$[ -1.618, -0.618 ] \cup [ 0.445, 1.802].$
\end{theorem}

\begin{proof}
    See the proofs of \cite[Theorem 3.3]{bg22} or \cite[Theorem 5.3]{pp21}.
\end{proof}

\subsubsection{Block-triangular preconditioners for $\sT$}

Similar to the block-arrow case, we consider the upper triangular preconditioner obtained by taking $\mathcal{D}_t\mathcal{L}_t^T$ obtained from the block-LDL decomposition of $\sT$ \eqref{eq:tri_ldl}, which in this case gives
\begin{equation}
    \mathcal{P}_T^t := \mathcal{D}_t\mathcal{L}_t^T = \begin{bmatrix}
        A_1 & B_1^T & 0 \\
        0 & -S_{t,1} & B_2^T \\
        0 & 0 & S_{t,2}
    \end{bmatrix}.
\end{equation}

We can make the following observation.

\begin{proposition}[Eigenvalues of $ \sT (\mathcal{P}_T^t)^{-1}$]
    The right-preconditioned operator  $ \sT (\mathcal{P}_T^t)^{-1}$  has minimal polynomial $(\lambda - 1)^3$.  
\end{proposition}

\begin{proof}
The proof follows from the block-LDL factorization of $\sT$: we observe that 
\[
\sT  (\mathcal{P}_T^t)^{-1} = \sT (\mathcal{D}_t\mathcal{L}_t^T)^{-1} = \mathcal{L}_t,
\]
where
\[
\mathcal{L}_t = \begin{bmatrix}
    I & 0 & 0 \\
    B_1A_1^{-1} & I & 0 \\
    0 & -B_2 S_{t,1}^{-1} & I
\end{bmatrix}.
\]
We can verify  $(\mathcal{L}_t - I)^3 = 0$.
\end{proof}

Additional variations of block-triangular Schur complement-based preconditioners for $\sT$ are presented and analyzed in \cite{cjl21}, including for the case of nonsymmetric matrices. 

We also note that several works have analyzed Schur complement preconditioning for $\sT$ in settings where the Schur complements are applied inexactly. The first such analysis for double saddle‑point systems was carried out by Bradley \& Greif \cite{bg22}, who combined eigenvalue bounds for the unpreconditioned operator with spectral equivalence assumptions on approximate Schur complements to derive inclusion intervals for the preconditioned spectrum. More recently, Bergamaschi et. al. \cite{bmpp25} used perturbation analysis to improve on the interior bounds derived in \cite{bg22}. Similar analyses for block-triangular preconditioners was also completed by Bergamaschi et. al. \cite{BergamaschiTriangular2025} with a focus on linear systems arising from poromechanics, and for multiple saddle-point systems by Pilotto et. al. \cite{pbm26}

\subsection{Other preconditioning approaches}
\label{sec:prec_other}

While our discussion in this section has focused on Schur complement–based preconditioning strategies, we will note that other preconditioning approaches have also been developed. We restrict attention here to a brief overview intended to indicate the range of existing approaches, rather than to provide a detailed survey or analysis. Some of the techniques here are straightforward generalizations from  classical saddle-point systems and, while they do not necessarily exploit features unique to the double saddle-point structure, they nonetheless give rise to interesting avenues of research. We mention here constraint preconditioning or monolithic multigrid as prominent examples; applications to double saddle-point systems can be found, for example, in \cite{abcfmt21,ch24,rg13}.

{\em Operator preconditioning} is a powerful paradigm, closely related to Schur complement preconditioning \cite{MardalWinther2011}. It is done at the continuous operator level, on chosen function spaces. The preconditioners are defined via Riesz maps induced by an appropriately chosen norm. The goal is to obtain a preconditioned operator that is uniformly well-conditioned independently of mesh size and physical parameters. This is accomplished by making a judicious choice of spaces and using inf-sup stability on the chosen spaces.

Operator preconditioning is connected to Schur complement preconditioning in the sense that the eventual block preconditioner is often an effective approximation of an ideal Schur complement-based block preconditioner. Because it relies on the same underlying mechanisms, operator preconditioning can be applied to double saddle-point systems in a manner analogous to its use for classical saddle-point problems. It has proven to be extremely robust in some instances, e.g., for Stokes-Darcy \cite{hkm21}.

{\em Shift-splitting  preconditioners}
are inspired by a combination of mechanisms of stationary methods, Hermitian/Skew-Hermitian splitting, and other classical approaches. In its most basic form, one applies a regular splitting that includes a parameter-dependent shift:
given a positive parameter $\alpha > 0$ and the identity matrix $I$, we write, say for $\sT$:
\begin{equation*}
\sT = \underbrace{\tfrac{1}{2}(\alpha I + \sT)}_{\mathcal{M}} - \underbrace{\tfrac{1}{2}(\alpha I - \sT)}_{\mathcal N},
\end{equation*}
and refer to $\mathcal{M}$ as the part to be inverted, either as a stationary scheme or as a preconditioner for a Krylov subspace solver.

In the context of double saddle-point systems, the shift-splitting approach has been studied both for the block-arrow type and the block-tridiagonal type systems. Often the original matrix is modified before the preconditioner is applied; for example, it is common to negate the second and third block rows of $\sT$ and only then apply preconditioning. Much of the analysis of these methods focuses on identifying optimal or near-optimal choices of the shift parameter.

Several variations of the shift-splitting iteration and/or preconditioning schemes have been developed for double saddle-point systems. For block-tridiagonal systems, we refer to \cite{ak26, bhb24, c19,lxm25,lz24}. For splitting-based preconditioners for the block arrow form, see \cite{cr23,dl23,l24,lz19,mhl23,mlm21,rcw22}.

\section{Concluding remarks}
\label{sec:conclusion}

The numerical solution of double saddle-point systems offers considerable challenges, and the broad importance of these systems across a variety of applications makes them a timely topic of investigation. The $\st$ family represents a broad set of systems of interest. The classification framework we have provided covers a practically relevant family of such systems, encompassing both the block-arrow and block-tridiagonal forms. By viewing these matrices through the lens of constrained optimization and exploiting their block structure, we established common algebraic, spectral, and factorization properties that underpin a wide class of existing models and solvers.

\rv{
Example \ref{ex:SD}, following Theorem \ref{thm:sa_eig}, illustrates the potential limitations of  eigenvalue bounds: the pursuit of uniformly tight interior bounds often requires additional assumptions, which in turn require application-specific knowledge. 
In addition, it is curious to observe that  clustering properties of preconditioned block-arrow matrices seem to be inferior to the analogous preconditioned block-tridiagonal matrix when similar block-diagonal Schur complement-based preconditioners are used. Charaterizing this situation fully requires additional investigation. We speculate that these initial findings may be attributed to the fact that the block-arrow form is not block-narrow-banded, and therefore more terms of the (full) block-LDL decomposition are dropped when the preconditioner is formed, compared to the block-tridiagonal form. That said, the block-arrow form is very valuable and is often the most natural setting for the application at hand. While it is tempting to use a block permutation that forces  it into a block-tridiagonal form, this may backfire in terms of the design of block preconditioners: it may turn the (2,2) block into the largest diagonal block in the matrix, and may thus make the standard Schur complement rank deficient. This is, in part, why we have restricted the $\st$ family to have a leading block that is the largest of the diagonal blocks of the matrix.
}

A compelling direction for future research is the identification of  properties that are intrinsic to double and multiple saddle-point systems and are not simply inherited from classical saddle-point theory. It is desirable to identify mathematical and numerical properties that are unique to these systems in the hope that this can lead to solution methods that fully exploit those properties. Extending these ideas to nonsymmetric formulations, inexact solvers, and large-scale multiphysics applications represents a natural continuation of this framework. 

\subsection*{Acknowledgments} We thank Nihar Darbhamulla for the Stokes-Darcy code used in Example \ref{ex:SD}. We are grateful to two anonymous referees for their helpful remarks.

\bibliographystyle{plain}
\bibliography{references-NM}
\end{document}